\documentclass{tran-l}

\usepackage[left=2.8cm,right=2.8cm,top=3cm,bottom=3cm]{geometry}
\usepackage{amssymb,amsmath,amsfonts,xcolor,graphicx}
\usepackage{amsthm,float,mathrsfs}
\usepackage{setspace,cancel}
\usepackage[colorlinks=true, linkcolor=black, citecolor=black, urlcolor=black]{hyperref}
\usepackage{enumitem}

\newcommand{\R}{\mathbb{R}}
\newcommand{\Q}{\mathbb{Q}}
\newcommand{\Z}{\mathbb{Z}}
\newcommand{\N}{\mathbb{N}}

\newcommand{\V}{\mathcal{V}}

\newcommand{\cL}{\mathcal{L}}
\newcommand{\B}{\mathcal{B}}
\newcommand{\x}{\mathbf{x}}
\newcommand{\cH}{\mathcal{H}}
\newcommand{\cA}{\mathcal{A}}

\newcommand{\D}{\mathcal{D}}
\newcommand{\cC}{\mathcal{C}}

\newcommand{\cS}{\mathcal{S}}

\newcommand{\cR}{\mathcal{R}}
\newcommand{\p}{\mathbf{p}}

\newcommand{\y}{\mathbf{y}}

\newtheorem{theorem}{Theorem}[section]
\newtheorem{lemma}[theorem]{Lemma}
\newtheorem{proposition}[theorem]{Proposition}

\theoremstyle{definition}

\theoremstyle{remark}
\newtheorem{remark}[theorem]{Remark}

\numberwithin{equation}{section}

\begin{document}

\title{ON THE HAUSDORFF DIMENSION OF WEIGHTED BADLY APPROXIMABLE VECTORS}


\author{Yi Lou}
\address{DEPARTMENT OF MATHEMATICS, NATIONAL UNIVERSITY OF SINGAPORE, SINGAPORE}
\curraddr{}
\email{lou\_yi@u.nus.edu}
\thanks{}


\date{}

\dedicatory{}

\begin{abstract}
Let \(\boldsymbol{\tau}=(\tau_1,\dots,\tau_m)\in \R_{\ge 0}^m\) satisfy
$\sum_{i=1}^m \tau_i>1$ and  $\tau_1\ge \cdots \ge \tau_m$
Let \(\Psi_{\boldsymbol\tau}=(\psi_1,\dots,\psi_m)\) be given by
\[
\psi_i(q)=q^{-\tau_i}, \qquad i=1,\dots,m,
\]
and denote by \(\cA_m(\Psi_{\boldsymbol\tau})\) the set of \(\Psi_{\boldsymbol\tau}\)-approximable vectors in \([0,1]^m\). The associated set of weighted \(\Psi_{\boldsymbol\tau}\)-badly approximable vectors is defined by
\[
\B_m(\Psi_{\boldsymbol\tau})
=
\cA_m(\Psi_{\boldsymbol\tau})
\setminus
\bigcap_{0<c<1}\cA_m(c\Psi_{\boldsymbol\tau}).
\]
The main result of this paper is that, for every ball \(B\subseteq [0,1]^m\),
\[
\dim_{\cH}\bigl(B\cap \B_m(\Psi_{\boldsymbol\tau})\bigr)
=
\dim_{\cH}\cA_m(\Psi_{\boldsymbol\tau}).
\]
The proof extends the Cantor-type construction and mass distribution arguments of Koivusalo, Levesley, Ward, and Zhang from the unweighted to the weighted setting, and is independent of recent results on weighted exact approximation. 
\end{abstract}

\maketitle

\section{Introduction}
  Throughout, let $\N$ denote the set of nonnegative integers. Fix a dimension $m\geq 1$. For any function $\Psi=(\psi_1,\cdots,\psi_m):\N\rightarrow \R_{>0}^m$, define the set of \textbf{$\Psi$-approximable vectors in $[0,1]^m$}, $\cA_{m}(\Psi)$, to be
	\begin{align*}
	\cA_{m}(\Psi) := \{\x=(x_1,&\cdots,x_m)\in [0,1]^m\,:\,|qx_i-p_i|<\psi_i(q),\\&\text{for infinitely many}\,q\in \N,\,\p=(p_1,\cdots,p_m)\in \Z^m, \,\,i=1,\cdots,m\}.
	\end{align*}
Define the set of \textbf{weighted $\Psi$-badly approximable vectors} as
 $$\B_m(\Psi):=\cA_{m}(\Psi) \setminus \bigcap \limits_{0<c<1} \cA_{m}(c\Psi)$$
where $c\Psi$ denotes the function $(c\psi_1,\cdots,c\psi_m)$. Equivalently,
	$$\B_{m}(\Psi) := \cA_{m}(\Psi) \setminus \bigcap \limits_{\boldsymbol{c} \in \R_{>0}^m} \cA_{m}(\boldsymbol{c} \Psi)$$
where $\boldsymbol{c} \Psi$ denotes the function $(c_1\psi_1,\cdots,c_m\psi_m)$, for $\boldsymbol c=(c_1,\cdots,c_m)$.\\ 

In this article, we will study the set $\B_m(\Psi)$ and focus on the case where $\psi_i(q) = q^{-\tau_i}$ for some $\boldsymbol \tau=(\tau_1,\cdots,\tau_m)$. In this case, we denote such $\Psi$ by $\Psi_{\boldsymbol \tau}$, and correspondingly denote $\cA_{m}(\Psi)$ and $\B_{m}(\Psi)$ by $\cA_{m}(\Psi_{\boldsymbol \tau})$ and $\B_{m}(\Psi_{\boldsymbol \tau})$. Let $\sigma=\tau_1+\cdots+\tau_m$. When $\sigma<1$, the answer follows from the well-known Dirichlet's Theorem for Diophantine Approximation, which states the following:

\begin{theorem}[Dirichlet's Theorem]
Let $a_i\geq 1$ for $1\leq i \leq m$ and $a_1+\cdots+a_m=m+1$. For any $\x=(x_1,\cdots,x_m)\in \R^m$, there exist infinitely many integers $p_1,\cdots,p_m,q$ such that
	$$\left|x_i-\frac{p_i}{q}\right|< q^{-a_i}\,\,\text{for}\,\,i=1,\cdots,m$$
\end{theorem}

Set $\lambda = \frac{1}{m}(1-\sigma)$. Then for any fixed $c\in (0,1)$, when $q$ is large enough, $cq^{-\tau_i} \geq q^{-\tau_i-\lambda}$. Hence $$\cA_m(\Psi_{\boldsymbol \tau + \lambda})\subseteq\cA_m(c\Psi_{\boldsymbol \tau}).$$
By Dirichlet's Theorem, $\cA_m(\Psi_{\boldsymbol \tau+\lambda})=[0,1]^m$. This implies $\cA_m(c\Psi_{\boldsymbol \tau})=[0,1]^m$ for any $c\in (0,1)$, which implies that $\B_m(\Psi_{\boldsymbol \tau})=\emptyset$.
 
The case $\sigma=1$ has been studied extensively. Despite $\B_m(\boldsymbol \tau)$ being null in the sense that it has zero Lebesgue measure, Schmidt \cite{SCHMIDT1969139} showed that it has full Hausdorff dimension, that is, $\dim_{\cH} \B_m(\boldsymbol \tau) = m$.
  
For $\sigma > 1$, Rynne and Dickinson have computed in \cite{Rynne2000} the Hausdorff dimension of $\cA_m(\Psi_{\boldsymbol \tau})$ using the direct method. More precisely, if $\tau_1 \geq \tau_2 \geq \cdots \geq \tau_m > 0$ and $\sigma > 1$, then
\begin{equation}
\dim_{\cH} \cA_m (\Psi_{\boldsymbol \tau})
=
\min_{1 \leq i \leq m}
\left\{
\frac{m+1 + \sum_{k=i}^m (\tau_i - \tau_k)}{1+\tau_i}
\right\}.
\end{equation}
In \cite{wang2015mass}, Wang and Wu reached the same conclusion using Mass Transference Principle. In \cite{KoivusaloLevesleyWardZhang2024}, Koivusalo, Levesley, Ward and Zhang proved that if $\boldsymbol \tau = (\tau_1,\dots,\tau_m)$ with $\tau_1 = \cdots = \tau_m > \frac{1}{m}$, then for any ball $B \subseteq [0,1]^m$,
\[
\dim_{\cH}\bigl(B \cap \B_m(\Psi_{\boldsymbol \tau})\bigr)
=
\dim_{\cH} \cA_m(\Psi_{\boldsymbol \tau})
=
\frac{m+1}{1+\tau_1}.
\]
Furthermore, Bandi and De Saxc\'{e} showed in \cite{bandi2023}, using a different method, that the set of (unweighted) exact $\Psi_{\boldsymbol \tau}$-approximable vectors has the same Hausdorff dimension as $\cA_m(\Psi_{\boldsymbol \tau})$. That is, for $\boldsymbol \tau = (\tau_1,\dots,\tau_1)$, $\tau_1>\frac{1}{m}$,
\[
\dim_{\cH} E_m(\Psi_{\boldsymbol \tau})
=
\dim_{\cH} \cA_m(\Psi_{\boldsymbol \tau}),
\]
where 
$E_m(\Psi_{\boldsymbol \tau})
:=
\cA_m(\Psi_{\boldsymbol \tau})
\setminus
\bigcup_{0<c<1} \cA_m(c \Psi_{\boldsymbol \tau})$. Since
\[
E_m(\Psi_{\boldsymbol \tau})
\subseteq
\B_m(\Psi_{\boldsymbol \tau})
\subseteq
\cA_m(\Psi_{\boldsymbol \tau}),
\]
the result of Bandi and De Saxc\'{e} in \cite{bandi2023} provides an independent proof for
\[
\dim_{\cH} \B_m(\Psi_{\boldsymbol \tau})
=
\frac{m+1}{1+\tau_1}
\]
when $\boldsymbol \tau = (\tau_1,\dots,\tau_1)$, $\tau_1>\frac{1}{m}$.\\

  The aim of this article is to extend the result of Koivusalo et al.\ in \cite{KoivusaloLevesleyWardZhang2024} to the weighted setting by proving the following theorem:

\begin{theorem}\label{tb:1}
Let \(\boldsymbol\tau=(\tau_1,\dots,\tau_m)\in \R_{\geq 0}^m\) satisfy
\[
\sum_{i=1}^m \tau_i>1
\qquad\text{and}\qquad
\tau_1\geq \cdots \geq \tau_m.
\]
Then, for any ball \(B\subseteq [0,1]^m\),
\begin{equation}
\dim_{\cH}\bigl(B\cap \B_m(\Psi_{\boldsymbol\tau})\bigr)
=
\dim_{\cH}\cA_m(\Psi_{\boldsymbol\tau})
=
\min_{1\leq i\leq m}
\left\{
\frac{m+1+\sum_{k=i}^m(\tau_i-\tau_k)}{1+\tau_i}
\right\}
=: s.
\end{equation}
\end{theorem}
\begin{remark}
This article, and in particular this main theorem, forms part of my undergraduate final-year project. 
During the preparation of this manuscript, I became aware that Bandi and Fregoli had circulated a preprint \cite{bandi2026hausdorffdimensionweightedexactly} in which they prove that, for general $\boldsymbol \tau \in \R_{>0}^m$, if $\sigma= \sum_{i=1}^m \tau_i>1$, then the set of (weighted) exact $\Psi_{\boldsymbol \tau}$-approximable vectors has the same Hausdorff dimension as $\cA_m(\Psi_{\boldsymbol \tau})$. An immediate consequence of their result is that, for $\sigma>1$,
\[
\dim_{\cH} \B_m(\Psi_{\boldsymbol \tau})
=
\dim_{\cH} \cA_m(\Psi_{\boldsymbol \tau}).
\]
Thus, the work of Bandi and Fregoli establishes a result which is, in a certain sense, stronger than the one considered in this article. Nevertheless, the proof presented here is independent of their approach and instead emphasises the local structure underlying the problem.
\end{remark}

We first consider the case \(\tau_m=0\). In this case, it is easy to see that
\[
\B_m(\Psi_{\boldsymbol\tau})
=
\B_{m-1}\bigl(\Psi_{(\tau_1,\dots,\tau_{m-1})}\bigr)\times [0,1].
\]
Consequently, the statement of the theorem in this case reduces to the corresponding statement in dimension \(m-1\).

The case \(m=1\) for Theorem \ref{tb:1} was proved by Koivusalo, Levesley, Ward, and Zhang in \cite{KoivusaloLevesleyWardZhang2024}, since in one dimension there is no distinction between weighted and unweighted approximation. So by an induction argument, we only need to consider the case \(\tau_m>0\) for Theorem \ref{tb:1}.

We claim that, in this case, Theorem~\ref{tb:1} follows directly from the following result.
\begin{theorem}\label{tb:2}
Suppose that \(\boldsymbol{\tau}=(\tau_1,\dots,\tau_m)\in \R_{>0}^m\) satisfies
\[
\sum_{i=1}^m \tau_i>1
\qquad\text{and}\qquad
\tau_1\ge \tau_2\ge \cdots \ge \tau_m.
\]
Let \(B\subseteq [0,1]^m\), let \(\epsilon>0\), and let \(\boldsymbol{\rho}=(\rho_1,\dots,\rho_m)\in \Q_{\ge 1}^m\) satisfy
\[
\sum_{i=1}^m \rho_i=m+1,\qquad \rho_1\ge \cdots \ge \rho_m,
\qquad\text{and}\qquad
1+\tau_i-\rho_i>0 \quad \text{for } i=1,\dots,m.
\]
Let
\[
\mathcal A:=\{1+\tau_i,\rho_i:i=1,\dots,m\}.
\]
For each \(A\in \mathcal A\), let \(\nu_1\in\{1,\dots,m\}\) be the integer such that
\[
A\le \rho_{\nu_1}\quad\text{but}\quad A>\rho_{\nu_1+1};
\]
if no such integer exists, set \(\nu_1=0\). Likewise, let \(\nu_2\in\{1,\dots,m\}\) be the integer such that
\[
A\ge 1+\tau_{\nu_2}\quad\text{but}\quad A<1+\tau_{\nu_2-1};
\]
if no such integer exists, set \(\nu_2=m+1\). Define
\begin{equation}\label{eqb:16}
\hat{s}_{\boldsymbol\rho}
:=
\min_{A\in \mathcal A}
\left\{
\sum_{i=1}^{\nu_1}1
+
\sum_{i=\nu_2}^m 1
-
\sum_{i=\nu_2}^m \frac{1+\tau_i-\rho_i}{A}
+
\sum_{i=\nu_1+1}^{\nu_2-1}\frac{\rho_i}{A}
\right\},
\end{equation}
and set
\[
s_{\boldsymbol\rho}:=\min\{\hat{s}_{\boldsymbol\rho},s\}.
\]
Then one may construct a Cantor subset
\[
\D_\epsilon(B)\subseteq B\cap \B_m(\Psi_{\boldsymbol\tau})
\]
such that
\[
\dim_{\cH}\D_\epsilon(B)\ge s_{\boldsymbol\rho}-2\epsilon.
\]
\end{theorem}

\begin{proof}[Proof of Theorem~\ref{tb:1}]
Fix \(B\subseteq [0,1]^m\) and \(\boldsymbol{\tau}=(\tau_1,\dots,\tau_m)\in \R_{>0}^m\). The upper bound
\[
\dim_{\cH}\bigl(B\cap \B_m(\Psi_{\boldsymbol\tau})\bigr)\le s
\]
follows from the inclusion
\[
B\cap \B_m(\Psi_{\boldsymbol\tau})\subseteq \cA_m(\Psi_{\boldsymbol\tau})
\]
together with the identity
\[
\dim_{\cH}\cA_m(\Psi_{\boldsymbol\tau})=s.
\]
It remains to prove the lower bound
\[
\dim_{\cH}\bigl(B\cap \B_m(\Psi_{\boldsymbol\tau})\bigr)\ge s.
\]

In \cite{wang2021masstransferenceprinciplerectangles}, Wang and Wu gave an alternative characterization of the dimension \(s\). Namely, setting
\[
a_i=\rho_i,\qquad t_i=1+\tau_i-\rho_i,\qquad i=1,\dots,m,
\]
and
\[
\mathcal A=\{a_i,a_i+t_i:i=1,\dots,m\},
\]
one has
\begin{equation}\label{eqb:15}
s\ge
\min_{A\in \mathcal A}
\left\{
\sum_{i\in \mathcal K_1}1
+
\sum_{i\in \mathcal K_2}1
-
\sum_{i\in \mathcal K_2}\frac{t_i}{A}
+
\sum_{i\in \mathcal K_3}\frac{\rho_i}{A}
\right\},
\end{equation}
where
\[
\mathcal K_1:=\{i:a_i\ge A\},\qquad
\mathcal K_2:=\{i:a_i+t_i\le A\}\setminus \mathcal K_1,
\qquad
\mathcal K_3:=\{1,\dots,m\}\setminus (\mathcal K_1\cup \mathcal K_2).
\]
Under the assumptions
\[
\rho_1\ge \cdots \ge \rho_m
\qquad\text{and}\qquad
1+\tau_i-\rho_i>0 \quad (i=1,\dots,m),
\]
the index sets \(\mathcal K_1,\mathcal K_2,\mathcal K_3\) become contiguous blocks, and \eqref{eqb:15} coincides with \eqref{eqb:16}.

We now distinguish two cases.

\begin{enumerate}
\item \textbf{Case 1: \(\tau_m\ge \frac1m\).}

In this case, it was shown in \cite{wang2021masstransferenceprinciplerectangles} that if we take
\[
\rho_i=1+\frac1m,\qquad
t_i=1+\tau_i-\rho_i,
\qquad i=1,\dots,m,
\]
then
\[
s
=
\min_{A\in \mathcal A}
\left\{
\sum_{i\in \mathcal K_1}1
+
\sum_{i\in \mathcal K_2}1
-
\sum_{i\in \mathcal K_2}\frac{t_i}{A}
+
\sum_{i\in \mathcal K_3}\frac{\rho_i}{A}
\right\}
=
\hat{s}_{\boldsymbol\rho}.
\]

Since \(\tau_1\ge \cdots \ge \tau_m\) and \(\sum_{i=1}^m\tau_i>1\), we can choose a sequence
\[
\{\boldsymbol\rho_k=(\rho_{1,k},\dots,\rho_{m,k})\}_{k\in \N}
\]
such that, for every \(k\in \N\),
\[
\rho_{i,k}\in \Q,\qquad \rho_{i,k}\ge 1,\qquad
\sum_{i=1}^m \rho_{i,k}=m+1,\qquad
1+\tau_i-\rho_{i,k}>0,
\]
\[
\rho_{1,k}\ge \cdots \ge \rho_{m,k},
\qquad\text{and}\qquad
\rho_{i,k}\to 1+\frac1m \quad \text{as } k\to\infty
\quad (i=1,\dots,m).
\]
Let \(\epsilon_k:=1/k\). Applying Theorem~\ref{tb:2} with \(\boldsymbol\rho=\boldsymbol\rho_k\) and \(\epsilon=\epsilon_k\), we obtain a Cantor subset
\[
\D_{\epsilon_k}(B)\subseteq B\cap \B_m(\Psi_{\boldsymbol\tau})
\]
such that
\[
\dim_{\cH}\D_{\epsilon_k}(B)\ge \min\{\hat{s}_{\boldsymbol\rho_k},s\}-\frac{2}{k}.
\]
Hence
\[
\dim_{\cH}\bigl(B\cap \B_m(\Psi_{\boldsymbol\tau})\bigr)
\ge
\min\{\hat{s}_{\boldsymbol\rho_k},s\}-\frac{2}{k}.
\]
Letting \(k\to\infty\), and using continuity of \(\hat{s}_{\boldsymbol\rho}\) in \(\boldsymbol\rho\), we obtain
\[
\dim_{\cH}\bigl(B\cap \B_m(\Psi_{\boldsymbol\tau})\bigr)
\ge
\min\{\hat{s}_{(1+1/m,\dots,1+1/m)},s\}
=
\min\{s,s\}
=
s.
\]

\item \textbf{Case 2: \(\tau_m<\frac1m\).}

Let \(K\) be the largest integer such that
\[
\tau_K>\frac{1-(\tau_{K+1}+\cdots+\tau_m)}{K}.
\]
Define
\[
\rho_i=
\begin{cases}
1+\dfrac{1-(\tau_{K+1}+\cdots+\tau_m)}{K}, & 1\le i\le K,\\[1em]
1+\tau_i, & K+1\le i\le m,
\end{cases}
\]
and let
\[
t_i=1+\tau_i-\rho_i,\qquad i=1,\dots,m.
\]
Then
\[
s
=
\min_{A\in \mathcal A}
\left\{
\sum_{i\in \mathcal K_1}1
+
\sum_{i\in \mathcal K_2}1
-
\sum_{i\in \mathcal K_2}\frac{t_i}{A}
+
\sum_{i\in \mathcal K_3}\frac{\rho_i}{A}
\right\}
=
\hat{s}_{\boldsymbol\rho}.
\]

Now choose a sequence
\[
\{\boldsymbol\rho_k=(\rho_{1,k},\dots,\rho_{m,k})\}_{k\in \N}
\]
such that
\[
\rho_{i,k}\in \Q,\qquad \rho_{i,k}\ge 1,\qquad
\sum_{i=1}^m \rho_{i,k}=m+1,\qquad
1+\tau_i-\rho_{i,k}>0,
\]
\[
\rho_{1,k}\ge \cdots \ge \rho_{m,k},
\qquad\text{and}\qquad
\rho_{i,k}\to \rho_i \quad \text{as } k\to\infty
\quad (i=1,\dots,m).
\]
Letting again \(\epsilon_k:=1/k\), and applying Theorem~\ref{tb:2} with \(\boldsymbol\rho=\boldsymbol\rho_k\) and \(\epsilon=\epsilon_k\), we obtain
\[
\dim_{\cH}\bigl(B\cap \B_m(\Psi_{\boldsymbol\tau})\bigr)
\ge
\min\{\hat{s}_{\boldsymbol\rho_k},s\}-\frac{2}{k}.
\]
Passing to the limit and using continuity of \(\hat{s}_{\boldsymbol\rho}\), we deduce that
\[
\dim_{\cH}\bigl(B\cap \B_m(\Psi_{\boldsymbol\tau})\bigr)\ge s.
\]
\end{enumerate}

Combining this with the upper bound completes the proof of Theorem~\ref{tb:1}.
\end{proof}

We devote the remainder of the article to the proof of Theorem~\ref{tb:2}. The argument adapts the proof of the unweighted case due to Koivusalo et al.~\cite{KoivusaloLevesleyWardZhang2024}. The proof proceeds in three steps.
\begin{itemize}
\item \textbf{Step 1. Construction of \(\cC^{\boldsymbol\tau}(N)\).}
Given \(\boldsymbol\rho=(\rho_1,\cdots,\rho_m)\in \Q_{\geq 1}^m\) such that
\[
\sum_{i=1}^m \rho_i=m+1,\qquad \rho_1\geq \cdots\geq \rho_m,\qquad \text{and}\qquad 1+\tau_i-\rho_i>0 \,\,\,\text{for }\,\,i=1,\cdots,m,
\]
we choose a suitable integer \(N\in \N\) and construct a set \(\cC^{\boldsymbol\tau}(N)\), together with a vector \(\boldsymbol c(N)=(c_1(N),\cdots,c_m(N))\), with the property that for every rational point $
\frac{\mathbf p}{q}=\left(\frac{p_1}{q},\cdots,\frac{p_m}{q}\right)\in \Q^m\cap [0,1]^m$,
one has
\[
\cC^{\boldsymbol\tau}(N)\cap \prod_{i=1}^m
\left[
\frac{p_i}{q}-c_i(N)q^{-1-\tau_i},
\frac{p_i}{q}+c_i(N)q^{-1-\tau_i}
\right]
=\emptyset.
\]
The set \(\cC^{\boldsymbol\tau}(N)\) is obtained by successively removing neighbourhoods of suitable rational points. In particular, if \(\mathbf x\in \cC^{\boldsymbol\tau}(N)\), then $\mathbf x\notin \cA_m(\boldsymbol c(N)\Psi_{\boldsymbol\tau}).$

\item \textbf{Step 2. Definition of the leading rationals.}
The construction of \(\cC^{\boldsymbol\tau}(N)\) naturally determines a distinguished subset of \(\Q^m\), which we call the \emph{leading rationals}. We shall establish two key structural properties of this set. These leading rationals form the basis for the later construction of \(\D_\epsilon(B)\).

\item \textbf{Step 3. Construction of \(\D_\epsilon(B)\).}
For any ball \(B\subseteq [0,1]^m\) and any fixed \(\epsilon>0\), we use the set of leading rationals to construct a Cantor set \(\D_\epsilon(B)\) such that
\[
\D_\epsilon(B)\subseteq B\cap \cC^{\boldsymbol\tau}(N)\cap \cA_m(\Psi_{\boldsymbol\tau}).
\]
Since points in \(\cC^{\boldsymbol\tau}(N)\) do not lie in \(\cA_m(\boldsymbol c(N)\Psi_{\boldsymbol\tau})\), it follows that
\[
\D_\epsilon(B)\subseteq B\cap \B_m(\Psi_{\boldsymbol\tau}).
\]
The key tool in this construction is the \(T_{G,\cR}\)-Lemma, which is inspired by the \(K_{G,B}\)-covering lemma in \cite{BeresnevichVelani2006} and the \(T_{G,I}\)-lemma in \cite{KoivusaloLevesleyWardZhang2024}. Finally, we show that \(\D_\epsilon(B)\) has sufficiently large Hausdorff dimension by applying the following mass distribution principle.
\end{itemize}

\begin{proposition}[Mass Distribution Principle]\label{pb:1}
Let $\mu$ be a probability measure supported on a subset $X$ of $\R^m$. Suppose there are positive constants $\alpha$, $C$ and $r_o$ such that
$$\mu(F)\leq Cr(F)^{\alpha}$$
for any ball $F$ with radius $r(F)\leq r_o$. Then if $E$ is a subset of $X$ with $\mu(E)=u>0$, then $\cH^{\alpha}(E)\geq uC^{-1}$, where $\cH^{\alpha}(E)$ denotes the $\alpha$-dimensional Hausdorff measure of $E$. Consequently, $$\dim_{\cH} E\geq \alpha.$$
\end{proposition}

\begin{itemize}[resume]
\item[] By the Mass Distribution Principle, to prove Theorem \ref{tb:2}, it suffices to define a mass distribution/probability measure $\mu$ on $\D_\epsilon(B)$ and verify that $\mu$ satisfies the hypothesis of the proposition with exponent $\alpha = s_{\boldsymbol \rho}-2\epsilon$.
\end{itemize}

\vspace{0.3cm}
We also introduce some of the notations and conventions that will be used throughout the thesis:
\begin{itemize}
\item Unless otherwise stated, the metric used in the article is the metric induced by the standard supremum norm.
\item Suppose $x$ is a real number, then $\lceil x\rceil$ denotes the smallest integer no less than $x$.
\item Suppose $S$ is a set, then $\#S$ denotes the cardinality of $S$. If $S$ is a collection of subsets of $\R^m$, then $\bigcup S$ denotes the union of those elements in $S$.
\item Suppose $E\subseteq R^m$, then $\lambda_m(E)$ denotes the $m$-dimensional Lebesgue measure of $E$. 
\item Suppose $E\subseteq [0,1]^m$, and $\boldsymbol \epsilon = (\epsilon_1,\cdots,\epsilon_m)$ is a nonnegative vector. We define the $\boldsymbol \epsilon$-neighborhood of $S$, denoted $\Delta(E,\boldsymbol \epsilon)$, to be the set $E+\prod_{i=1}^m [-\epsilon_i,+\epsilon_i]$, where the addition is understood in the sense of Minkowski sums. It is then clear that $(x_1,\cdots,x_m) \in \Delta(E,\boldsymbol \epsilon)$ if and only if there exists some $(y_1,\cdots,y_m)\in E$, $|x_i-y_i|\leq \epsilon_i$, $i=1,\cdots,m$. If $E=\{z\}$ is a singleton, then we simply write $\Delta(z,\boldsymbol \epsilon)$ instead of $\Delta(E,\boldsymbol \epsilon)$. If $E=\emptyset$, then any $\boldsymbol \epsilon$-neighborhood of $E$ is set to be the empty set as well.
\item Suppose $\boldsymbol \xi=(\xi_1,\cdots,\xi_m)\in \R^m$, then we shall denote $q^{\boldsymbol \xi}$ to be the vector $(q^{\xi_1},\cdots,q^{\xi_m})$.
\item Suppose $R =\Delta \left(\x,\boldsymbol \epsilon\right)$, then $R$ is a rectangle in $\R^m$. $\x$ is said to be the center of the rectangle, and the components of $\boldsymbol \epsilon$ are called the side lengths of the rectangle. For any $a>0$, we shall denote $aR = \Delta \left(\x,a\boldsymbol \epsilon \right)$.  If $\x$ happens to be a rational point, i.e., $\x=\left(\frac{p_1}{q},\cdots,\frac{p_m}{q}\right)$, then we use $d(R)$ to denote the denominator of the center of $R$ which is just $q$.
\end{itemize}

\vspace{0.5cm}
\section{Step 1: The Set $\cC^{\boldsymbol \tau}(N)$}
The first step in the proof of Theorem~\ref{tb:2} is to construct a set $\cC^{\boldsymbol\tau}(N)\subseteq [0,1]^m$ with the property that, for every rational point \(\frac{\mathbf p}{q}\in [0,1]^m\cap \Q^m\), the set \(\cC^{\boldsymbol\tau}(N)\) avoids a suitable neighbourhood of \(\frac{\mathbf p}{q}\). The construction of \(\cC^{\boldsymbol\tau}(N)\) relies on the following Simplex Lemma, whose proof may be found in \cite{KristensenThornVelani2006}.
\begin{lemma}[Simplex Lemma]\label{lb:1}
Let $m\geq 1$ be an integer and $Q\in \N$. Let $E\subseteq \R^m$ be a convex set that satisfies
$$\lambda_m(E) \leq \frac{1}{m!} \cdot Q^{-m-1}$$
Suppose $E$ contains $m+1$ rational points with denominator $1\leq q \leq Q$. Then these rational points lie on some hyperplane of $\R^m$ intersected with $E$.
\end{lemma}

\medskip
\noindent\textbf{Construction of \(\cC^{\boldsymbol\tau}(N)\).}
\begin{enumerate}
\item Fix $\boldsymbol \rho=(\rho_1,\cdots,\rho_m) \in \Q_{\geq 1}^m$ such that 
	$$\sum\limits_{i=1}^m \rho_i = m+1,\quad \rho_1\geq \rho_2\geq \cdots \geq \rho_m,\quad \text{and}\quad 1+\tau_i-\rho_i>0,\,\,\text{for}\,\,i=1,\cdots,m.$$
We define $\cC^{\boldsymbol \tau}(N)$ for any $N\in \N$ such that $N^{\rho_i}$ is an integer for any $i\in \{1,\cdots,m\}$. Since $\boldsymbol \rho\in \Q_{\geq 1}^m$, there are infinitely many such choices of $N$.
\item
Fix an arbitrary integer \(t>\min\{2,m\}\) such that $t^{\rho_i}\in \N$ for any $i\in \{1,\cdots,m\}$. Note that \(t^{m+1}= m^{m+1}>m!\). For $\mathbf{n}= (n_1,\cdots,n_m) \in \N_{\geq 1}^m$, let \(\V_{\mathbf{n}}\) denote the collection of rectangles obtained by partitioning the unit cube $[0,1]^m$ into rectangles of side lengths
\[
t^{-\rho_1}N^{-n_1\rho_1},\,\,t^{-\rho_2}N^{-n_2\rho_2},\,\,\cdots,\,\,
t^{-\rho_m}N^{-n_m\rho_m},
\]
respectively. Equivalently, \(\V_{\mathbf{n}}\) is obtained by dividing the \(i\)-th coordinate interval into \(t^{\rho_i}N^{n_i\rho_i}\) equal subintervals for \(i=1,\cdots,m\).

If $\mathbf{n},\mathbf{n}'\in \N_{\geq 1}^m$ is such that $\mathbf{n}'\preceq \mathbf{n}$, i.e., $n'_i \leq n_i$ for all $i\in \{1,\cdots,m\}$, since $t^{\rho_i}$'s and $N^{\rho_i}$'s are integers, we see that $\V_{\mathbf{n}}$ is a refinement of $\V_{\mathbf{n}'}$. This is to say that for any $\cR\in \V_{\mathbf{n}}$, there exists a unique rectangle $\widetilde{\cR}\in \V_{\mathbf{n}'}$ such that $\cR\subseteq \widetilde{\cR}$. For notational simplicity, if \(\widetilde{\cR}\in \V_{\mathbf{n}'}\), we denote
\[
\V_{\mathbf{n}}(\widetilde{\cR})
:=
\{\cR\in \V_{\mathbf{n}}:\cR\subseteq \widetilde{\cR}\}.
\]
\item
For \(n\geq 1\), we have
\[
\#\V_{(n,\cdots,n)} = \prod\limits_{i=1}^m t^{\rho_i} N^{n\rho_i} = t^{m+1}N^{n(m+1)} \geq m!(N^n)^{m+1}
\]
Hence, for every rectangle \(\cR\in \V_{(n,\cdots,n)}\),
\[
\lambda_m(\cR)=t^{-m-1}N^{-n(m+1)}<\frac{1}{m!}(N^n)^{-m-1}.
\]
By Lemma~\ref{lb:1}, all rational points \(\frac{\mathbf p}{q}\in \cR\cap \Q^m\) with \(1\leq q<N^n\) lie on a $(m-1)-$dimensional hyperplane intersected with \(\cR\). If no such rational point exists, set \(L_{\cR}=\emptyset\). Otherwise, let \(L_{\cR}\) be the intersection of \(\cR\) with any $(m-1)$-dimensional hyperplane containing all such rational points. Define
\[
\cL_n:=\{L_{\cR}:\cR\in \V_{(n,\cdots,n)}\}.
\]
By abuse of terminology, we shall also refer to the elements of \(\cL_n\) as hyperplanes. Observe that if \(\frac{\mathbf p}{q}\in \Q^m\cap [0,1]^m\) and \(N^{n-1}\leq q<N^n\), then there exists some \(L\in \cL_n\) such that \(\frac{\mathbf p}{q}\in L\).

\item
Define \(\boldsymbol c(N)=(c_1(N),\cdots,c_m(N))\) by
\[
c_i(N):=t^{-\rho_i}N^{-3(1+\tau_i)},
\qquad i=1,\cdots,m.
\]
For \(n\geq 1\), define $\boldsymbol{\ell}(n)=(\ell_1(n),\cdots,\ell_m(n))$ by
\[
\ell_i(n):=\left\lceil (n+2)\frac{1+\tau_i}{\rho_i}\right\rceil,
\qquad i=1,\cdots,m.
\]
and define \(\boldsymbol\delta(n)=(\delta_1(n),\cdots,\delta_m(n))\) by
\[
\delta_i(n):=t^{-\rho_i}N^{-(n+2)(1+\tau_i)},
\qquad i=1,\cdots,m.
\]

\item
We now remove neighbourhoods of the hyperplanes iteratively. First, remove all rectangles in \(\V_{\boldsymbol{\ell}(1)}\) that intersect
\[
\bigcup_{L\in \cL^*_1}\Delta(L,\boldsymbol\delta(1)), \quad\text{where}\quad \cL^*_1=  \cL_1
\]
and let \(\cS_1\) denote the collection of surviving rectangles. More precisely,
\[
\cS_1
:=
\left\{
\cR\in \V_{\boldsymbol{\ell}(1)}
:
\cR\cap \bigcup_{L\in \cL^*_{1}}\Delta(L,\boldsymbol\delta(1))=\emptyset
\right\}.
\]
For \(n\geq 2\), suppose that \(\cS_j\) has been constructed for \(j=1,\dots,n-1\). Define
\begin{equation}\label{eqb:3}
\cL_n^*
:=
\left\{
L\in \cL_n
:
\exists \frac{\mathbf p}{q}\in L
\text{ such that }
\begin{cases}
\mathrm{(i)} & N^{n-1}\leq q<N^n,\\[0.3em]
\mathrm{(ii)} & \Delta\!\left(\frac{\mathbf p}{q},\boldsymbol\delta(n)\right)
\cap
\displaystyle\bigcup_{\cR\in \cS_{n-1}}\cR
\neq \emptyset
\end{cases}
\right\}.
\end{equation}
From $\cS_{n-1}$, remove all rectangles in \(\V_{\boldsymbol{\ell}(n)}\) that intersect the \(\boldsymbol\delta(n)\)-neighbourhood of some \(L\in \cL_n^*\). That is, define
\[
\cS_n
:=
\left\{
\cR\in \V_{\boldsymbol{\ell}(n)}
:
\cR\cap \bigcup_{L\in \cL_n^*}\Delta(L,\boldsymbol\delta(n))=\emptyset,
\quad
\exists \cR'\in \cS_{n-1}\text{ such that }\cR\subseteq \cR'
\right\}.
\]

\item
Finally, define \(\cC^{\boldsymbol\tau}(N)\) by
\[
\cC^{\boldsymbol\tau}(N)
:=
\bigcap_{n=1}^\infty \bigcup_{\cR\in \cS_n}\cR.
\]
This completes the construction of \(\cC^{\boldsymbol\tau}(N)\).
\end{enumerate}

\vspace{0.3cm}
We prove that \(\cC^{\boldsymbol\tau}(N)\) has the desired avoidance property.
\begin{lemma}\label{lb:2}
For any
\[
\frac{\mathbf p}{q}=\left(\frac{p_1}{q},\cdots,\frac{p_m}{q}\right)\in \Q^m\cap [0,1]^m,
\]
one has
\[
\cC^{\boldsymbol\tau}(N)\cap
\Delta\!\left(\frac{\mathbf p}{q},\,\mathbf c(N)\,q^{-1-\boldsymbol\tau}\right)
=\emptyset.
\]
\end{lemma}

\begin{proof}
Observe that \(\cC^{\boldsymbol\tau}(N)\) is defined as the intersection of a nested sequence of sets, and at each stage of the construction we ensure that
\[
\bigcup_{\cR\in \cS_n}\cR
\;\cap\;
\bigcup_{L\in \cL_n^*}\Delta(L,\boldsymbol\delta(n))
=\emptyset.
\]
Consequently,
\begin{equation}\label{eqb:1}
\cC^{\boldsymbol\tau}(N)
\cap
\bigcup_{n=1}^\infty \bigcup_{L\in \cL_n^*}\Delta(L,\boldsymbol\delta(n))
=
\emptyset.
\end{equation}

Now let
\[
\frac{\mathbf p}{q}\in \Q^m\cap [0,1]^m
\qquad\text{with}\qquad
N^{n-1}\leq q<N^n
\]
for some \(n\in \N\). Then, for \(i\in\{1,\cdots,m\}\),
\begin{equation}\label{eqb:2}
c_i(N)q^{-1-\tau_i}
\leq
c_i(N)N^{-(n-1)(1+\tau_i)}
=
t^{-\rho_i}N^{-(n+2)(1+\tau_i)}
=
\delta_i(n).
\end{equation}
By construction, there exists a hyperplane \(\widetilde L\in \cL_n\) containing \(\frac{\mathbf p}{q}\).

If \(\widetilde L\in \cL_n^*\), then by \eqref{eqb:1} and \eqref{eqb:2},
\[
\Delta\!\left(\frac{\mathbf p}{q},\,\mathbf c(N)\,q^{-1- \boldsymbol \tau}\right)
\subseteq
\Delta(\widetilde L,\boldsymbol\delta(n)),
\]
and hence
\[
\Delta\!\left(\frac{\mathbf p}{q},\,\mathbf c(N)\,q^{-1- \boldsymbol \tau}\right)
\cap
\cC^{\boldsymbol\tau}(N)
=
\emptyset.
\]

Suppose instead that \(\widetilde L\notin \cL_n^*\). Then, $n>1$ and by the definition of \(\cL_n^*\) in \eqref{eqb:3},
\[
\Delta(\widetilde L,\boldsymbol\delta(n))
\cap
\bigcup_{\cR\in \cS_{n-1}}\cR
=
\emptyset.
\]
Since \(\cC^{\boldsymbol\tau}(N)\subseteq \bigcup_{\cR\in \cS_{n-1}}\cR\), it follows that
\[
\Delta\!\left(\frac{\mathbf p}{q},\,\mathbf c(N)\,q^{-1-\boldsymbol\tau}\right)
\cap
\cC^{\boldsymbol\tau}(N)
\subseteq
\Delta(\widetilde L,\boldsymbol\delta(n))
\cap
\bigcup_{\cR\in \cS_{n-1}}\cR
=
\emptyset.
\]
This completes the proof.
\end{proof}

\medskip
Recall that the Cantor set \(\D_\epsilon(B)\) we wish to construct will be contained in
\[
B\cap \cC^{\boldsymbol\tau}(N)\cap \cA_2(\Psi_{\boldsymbol\tau}).
\]
It is therefore important to ensure that, for an arbitrary ball \(B\subseteq [0,1]^m\), the intersection
\[
B\cap \cC^{\boldsymbol\tau}(N)
\]
is non-empty for at least some values of \(N\). The existence of such \(N\) is guaranteed by the following lemma.
\begin{lemma}\label{lb:3}
For any \(\kappa>0\), there exists \(N_\kappa\in \N\) such that, whenever \(N\geq N_\kappa\),
\[
\lambda_m\bigl(\cC^{\boldsymbol\tau}(N)\bigr)\geq 1-\kappa.
\]
\end{lemma}

\begin{proof}
From the construction of \(\cC^{\boldsymbol\tau}(N)\), we have
\[
\lambda_m\bigl(\cC^{\boldsymbol\tau}(N)\bigr)
\geq
1-
\lambda_m\left(
\bigcup_{n\in \N}
\bigcup
\left\{
\cR\in \V_{\boldsymbol{\ell}(n)}
:
\cR\cap \bigcup_{L\in \cL_n}\Delta(L,\boldsymbol\delta(n))\neq \emptyset
\right\}
\right).
\]
The rectangles in \(\V_{\boldsymbol{\ell}(n)}\) have side lengths
\[
t^{-\rho_i}N^{-\ell_i(n)\rho_i},
\qquad i=1,\cdots,m.
\]

We claim that if \(\cR\in \V_{\boldsymbol{\ell}(n)}\) intersects \(\Delta(L,\boldsymbol\delta(n))\) for some \(L\in \cL_n\), then
\[
\cR\subseteq \Delta(L,2\boldsymbol\delta(n)).
\]
Indeed, by assumption there exist \(\mathbf x\in \cR\) and \(\mathbf y\in L\) such that
\[
|x_i-y_i|\leq \delta_i(n),\qquad i=1,\cdots,m.
\]
Then for any \(\mathbf z\in \cR\) and \(i\in \{1,\cdots,m\}\),
\begin{align*}
|z_i-y_i|
&\leq |z_i-x_i|+|x_i-y_i| \leq t^{-\rho_i}N^{-\ell_i(n)\rho_i}+\delta_i(n) \\
&=
t^{-\rho_i}N^{-\left\lceil (n+2)\frac{1+\tau_i}{\rho_i}\right\rceil \rho_i}+\delta_i(n) \\
&\leq t^{-\rho_i}N^{-(n+2)(1+\tau_i)}+\delta_i(n)= 2\delta_i(n).
\end{align*}
This proves the claim.

Similarly, using the triangle inequality, one sees that if \(\widetilde{\cR}\in \V_{n,n}\) is such that \(L\subseteq \widetilde{\cR}\), then any rectangle \(\cR\in \V_{\boldsymbol{\ell}(n)}\) that intersects \(\Delta(L,\boldsymbol\delta(n))\) is contained in \(5\widetilde{\cR}\). This follows from the facts that
\[
\delta_i(n)=t^{-\rho_i}N^{-(n+2)(1+\tau_i)}<t^{-\rho_i}N^{-n\rho_i}
\]
and
\[
t^{-\rho_i}N^{-\ell_i(n)\rho_i}<t^{-\rho_i}N^{-n\rho_i},
\qquad i=1,\cdots,m,
\]
while the side lengths of rectangles in \(\V_{(k_1,\cdots,k_m)}\) are \(t^{-\rho_i}N^{-k_i\rho_i}\), $i=1,\cdots,m$.

Consequently, for any \(L\in \cL_n\),
\begin{equation}\label{eq:measure-cover}
\lambda_m\left(
\bigcup
\left\{
\cR\in \V_{\boldsymbol{\ell}(n)}
:
\cR\cap \Delta(L,\boldsymbol\delta(n))\neq \emptyset
\right\}
\right)
\leq
\lambda_m\bigl(\Delta(L,2\boldsymbol\delta(n))\cap 5\widetilde{\cR}\bigr),
\end{equation}
where \(\widetilde{\cR}\in \V_{(n,\cdots,n)}\) is the rectangle corresponding to \(L\).

Next we estimate \(\lambda_m\bigl(\Delta(L,2\boldsymbol\delta(n))\cap 5\widetilde{\cR}\bigr)\). If \(L=\emptyset\), then this quantity is zero. Otherwise, \(L=L_{\widetilde{\cR}}\) is the intersection of a $(m-1)$-dimensional hyperplane \(H_{\widetilde{\cR}}\) with the rectangle \(\widetilde{\cR}\), whose side lengths are \(t^{-\rho_i}N^{-n\rho_i}\). Assume that \(H_{\widetilde{\cR}}\) is determined by $a_1x_1+a_2x_2+\cdots+a_mx_m+b=0$, $a_1,\cdots,a_m\neq 0$. The cases where at least one of the $a_i$'s is zero is treated similarly. Then
\begin{align*}
\Delta(L_{\widetilde{\cR}},2\boldsymbol\delta(n))\cap 5\widetilde{\cR}
&\subseteq
\Delta(H_{\widetilde{\cR}},2\boldsymbol\delta(n))\cap 5\widetilde{\cR} \\
&\subseteq
\left\{
(x_1,\cdots,x_m)\in 5\widetilde{\cR}
:
|a_1x_1+\cdots+a_mx_m+b|\leq 2\sum\limits_{j=1}^m |a_j|\delta_j(n)
\right\}.
\end{align*}

We now estimate the Lebesgue measure of the latter set. If all coordinates except the $i$th coordinate $(x_i)$ are fixed, then the admissible values of $x_i$ lie in an interval of length at most  
    $$\frac{4}{|a_i|} \sum\limits_{k=1}^m |a_k|\delta_k(n).$$
While $(x_1,\cdots,x_m)$ need to lie in $5\widetilde{\cR}$ which has side lengths $5t^{-\rho_i}N^{-n\rho_i}$, then
\begin{align*}
\lambda_m\bigl(\Delta(L_{\widetilde{\cR}},2\boldsymbol\delta(n))\cap 5\widetilde{\cR}\bigr)
&\leq
\min\limits_{1\leq i \leq m} \left\{ \frac{4}{|a_i|} \left(\sum\limits_{k=1}^m |a_k|\delta_k(n)\right) \left(\prod\limits_{j\neq i} 5t^{-\rho_j}N^{-n\rho_j}\right)\right\}
\end{align*}
We choose \(i_0\) such that
\[
\frac{P_{i_0}}{|a_{i_0}|}
=\min_{1\le i\le m}\frac{P_i}{|a_i|}, \quad \text{where}\quad P_i = \prod\limits_{j\neq i} t^{-\rho_j} N^{-n\rho_j}.
\]
Then for every \(k\in \{1,\cdots,m\}\),
\[
\frac{P_{i_0}}{|a_{i_0}|}\le \frac{P_k}{|a_k|}.
\]
Since \(\delta_k(n)\ge 0\), multiplying by \(|a_k|\delta_k(n)\) and summing over \(k\) gives
\[
\frac{P_{i_0}}{|a_{i_0}|}\sum_{k=1}^m |a_k|\delta_k(n)
\le \sum_{k=1}^m \frac{P_k}{|a_k|}|a_k|\delta_k(n)
= \sum_{k=1}^m \delta_k(n)P_k.
\]
Therefore
\[
\min_{1\le i\le m}\left\{\frac{4}{|a_i|}\sum_{k=1}^m |a_k|\delta_k(n)\prod_{j\ne i}5t^{-\rho_j}N^{-n\rho_j}\right\}
\le 4\cdot 5^{m-1}\sum_{k=1}^m \left(\delta_k(n)\prod_{j\ne k} t^{-\rho_j}N^{-n\rho_j}\right).
\]
Let
\[
\varepsilon:=\min_{1\leq i \leq m}\{1+\tau_i-\rho_i,\,1+\tau_i\}>0.
\]
Then
\begin{align*}
\lambda_m\bigl(\Delta(L_{\widetilde{\cR}},2\boldsymbol\delta(n))\cap 5\widetilde{\cR}\bigr)&\leq
4\cdot 5^{m-1}\sum_{k=1}^m \left(\delta_k(n)\prod_{j\ne k} t^{-\rho_j}N^{-n\rho_j}\right) \\
&=4\cdot 5^{m-1}\sum_{k=1}^m \left(t^{-\rho_k} N^{-(n+2)(1+\tau_k)}\prod_{j\ne k} t^{-\rho_j}N^{-n\rho_j}\right)\\
&\leq 
4m \cdot 5^{m-1} \cdot t^{-m-1}\cdot N^{-n(m+1)} (N^{-\varepsilon})^n
\end{align*}

Finally, \(\#\cL_n=\#\V_{(n,\cdots,n)}=t^{m+1}N^{n(m+1)}\). Hence,
\begin{align*}
&\quad \lambda_m\left(
\bigcup_{n\in \N}
\bigcup
\left\{
\cR\in \V_{\boldsymbol{\ell}(n)}
:
\cR\cap \bigcup_{L\in \cL_n}\Delta(L,\boldsymbol\delta(n))\neq \emptyset
\right\}
\right)\\&\leq
\sum_{n=1}^\infty
\sum_{\widetilde{\cR}\in \V_{(n,\cdots,n)}}
\lambda_m\bigl(\Delta(L_{\widetilde{\cR}},2\boldsymbol\delta(n))\cap 5\widetilde{\cR}\bigr) \\
&\leq
\sum_{n=1}^\infty
\bigl(t^{m+1}N^{n(m+1)}\bigr)\cdot \bigl(4m \cdot 5^{m-1} \cdot t^{-m-1}\cdot N^{-n(m+1)} (N^{-\varepsilon})^n\bigr) \\
&=4m\cdot 5^{m-1}\cdot\sum_{n=1}^\infty (N^{-\varepsilon})^n.
\end{align*}
Since \(N^{-\varepsilon}\to 0\) as \(N\to \infty\), we may choose \(N_\kappa\) sufficiently large that
\[
4m\cdot 5^{m-1}\cdot \sum_{n=1}^\infty (N^{-\varepsilon})^n<\kappa
\qquad\text{for all }N\geq N_\kappa.
\]
Therefore,
\[
\lambda_m\bigl(\cC^{\boldsymbol\tau}(N)\bigr)\geq 1-\kappa
\qquad\text{whenever }N\geq N_\kappa.
\]
This completes the proof.
\end{proof}

\medskip
Next we establish the local statement for the mass removed during the construction of $\cC^{\boldsymbol \tau}(N)$.
\begin{lemma}\label{lb:4}
Suppose \(n\ge 1\) and \(\cR\in \cS_n\). For any \(\kappa>0\), there exists \(N_\kappa\in \N\) such that,
whenever \(N\ge N_\kappa\),
\[
\lambda_m\bigl(\cR\cap \cC^{\boldsymbol\tau}(N)\bigr)\ge (1-\kappa)\lambda_m(\cR).
\]
\end{lemma}

\begin{proof}
Fix \(n\ge 1\) and \(\cR\in \cS_n\). Since \(\cR\) survives up to level \(n\), we have
\[
\cR\subseteq \bigcup_{\cR'\in \cS_j}\cR'
\qquad\text{for every }1\le j\le n.
\]
Hence
\[
\lambda_m\!\left(
\cR\cap \bigcap_{j=1}^n \bigcup_{\cR'\in \cS_j}\cR'
\right)=\lambda_m(\cR).
\]
As \(\cC^{\boldsymbol\tau}(N)=\bigcap_{j=1}^\infty \bigcup_{\cR'\in \cS_j}\cR'\), it follows that
\begin{align*}
\lambda_m\bigl(\cR\cap \cC^{\boldsymbol\tau}(N)\bigr)
&\ge
\lambda_m(\cR)
-
\sum_{k=n+1}^\infty
\lambda_m\left(
\bigcup
\left\{
\cR'\in \V_{\boldsymbol{\ell}(k)}(\cR):
\cR'\cap \bigcup_{L\in \cL_k}\Delta(L,\boldsymbol\delta(k))\neq\emptyset
\right\}
\right) \\
&\ge
\lambda_m(\cR)
-
\sum_{k=n+1}^\infty
\sum_{\tilde{\cR}\in \V_{(k,\cdots,k)}}
\lambda_m\left(
\bigcup
\left\{
\cR'\in \V_{\boldsymbol{\ell}(k)}(\cR):
\cR'\cap \Delta(L_{\tilde{\cR}},\boldsymbol\delta(k))\neq\emptyset
\right\}
\right).
\end{align*}

We now estimate the inner sum for each \(k\ge n+1\). Write
\[
E_{\cR,\tilde{\cR},k}
:=
\bigcup
\left\{
\cR'\in \V_{\boldsymbol{\ell}(k)}(\cR):
\cR'\cap \Delta(L_{\tilde{\cR}},\boldsymbol\delta(k))\neq\emptyset
\right\}.
\]
Then $E_{\cR,\tilde{\cR},k} \subseteq \cR$. From the proof of Lemma \ref{lb:3}, we know $E_{\cR,\tilde{\cR},k} \subseteq \Delta( H_{\tilde{\cR}},2\boldsymbol \delta(k))$ and $E_{\cR,\tilde{\cR},k} \subseteq 5\tilde{\cR}$. Then 
    $$E_{\cR,\tilde{\cR},k} \subseteq \Delta(H_{\tilde{\cR}},2\boldsymbol \delta(k)) \cap W$$
where $W$ is a rectangle whose side lengths are bounded by
    $$\min \{t^{-\rho_i}N^{-\ell_i(n)\rho_i},\,5t^{-\rho_i} N^{-k\rho_i}\}\leq 5t^{-\rho_i} N^{-\max\{\ell_i(n),k\} \rho_i},\quad i=1,\cdots,m.$$
Similar to the proof of Lemma \ref{lb:3}, one can show
    \begin{align*}
    \lambda_m(E_{\cR,\tilde{\cR},k}) \leq M_1 \sum \limits_{i=1}^m \delta_i(k) \prod\limits_{j\neq i} t^{-\rho_j} N^{-\max\{\ell_j(n),k\} \rho_j}
    \end{align*}
for some constant $M_1$ only depending on $m$ and independent of $n$, $k$ and $N$.

We also give a bound on the number of $\tilde{\cR} \in \V_{(k,\cdots,k)}$ such that $E_{\cR,\tilde{\cR},k}$ is nonempty. Since $L_{\tilde{\cR}} \subseteq \tilde{\cR}$ which has side lengths $t^{-\rho_i}N^{-k\rho_i}$ and $\cR$ has sidelengths $t^{-\rho_i} N^{-\ell_i(n)\rho_i}\geq \delta_i(k)$ for $k>n$ and $i\in \{1,\cdots,m\}$, then a standard eometric argument gives that the number of $\tilde{\cR}$ such that $E_{\cR,\tilde{\cR},k}$ is nonempty is bounded by
    $$M_2 \prod\limits_{i=1}^m N^{\max\{k-\ell_i(n),0\}\rho_i}.$$
where $M_2$ is a constant depending only on $m$. Therefore,
\begin{align*}
&\quad \sum_{k=n+1}^\infty
\sum_{\tilde{\cR}\in \V_{(k,\cdots,k)}}
\lambda_m\left(
\bigcup
\left\{
\cR'\in \V_{\boldsymbol{\ell}(k)}(\cR):
\cR'\cap \Delta(L_{\tilde{\cR}},\boldsymbol\delta(k))\neq\emptyset
\right\}
\right)\\
&\leq \sum\limits_{k=n+1}^\infty \left(M_2 \prod\limits_{i=1}^m N^{\max\{k-\ell_i(n),0\}\rho_i}\right) \cdot \left(M_1 \sum \limits_{i=1}^m \delta_i(k) \prod\limits_{j\neq i} t^{-\rho_j} N^{-\max\{\ell_j(n),k\} \rho_i}\right)\\
&= M_1M_2t^{-m-1} \sum\limits_{k=n+1}^\infty \sum \limits_{i=1}^m \left( N^{-(k+2)(1+\tau_i)} N^{\max \{\ell_i(n),k\}\rho_i} \prod\limits_{j=1}^m N^{\max \{k-\ell_j(n),0\} \rho_j} N^{-\max \{\ell_j(n),k\} \rho_j} \right)\\
&= M_1M_2 t^{-m-1} N^{-\sum\limits_{i=1}^m \ell_i(n)\rho_i} \sum \limits_{k=n+1}^\infty \sum \limits_{i=1}^m \left( N^{-(k+2)(1+\tau_i)} N^{\max \{\ell_i(n),k\}\rho_i}\right)\\
&= M_1M_2 \lambda_m(\cR) \sum \limits_{k=n+1}^\infty \sum \limits_{i=1}^m \left( N^{-(k+2)(1+\tau_i)} N^{\max \{\ell_i(n),k\}\rho_i}\right)
\end{align*}
Since 
    $$\ell_i(n)\rho_i = \left\lceil (n+2)\frac{1+\tau_i}{\rho_i} \right\rceil \rho_i\leq (n+2)(1+\tau_i)+\rho_i$$
and $k\geq n+1$. Then by letting $\varepsilon = \min \limits_{1\leq i \leq m} \{1+\tau_i-\rho_i,1\}>0$, we see that 
$$\sum \limits_{k=n+1}^\infty \sum \limits_{i=1}^m \left( N^{-(k+2)(1+\tau_i)} N^{\max \{\ell_i(n),k\}\rho_i}\right) \leq \sum\limits_{k=n+1}^\infty \sum \limits_{i=1}^m (N^{-\varepsilon})^{k-n} = m\sum\limits_{k=n+1}^\infty (N^{-\varepsilon})^{k-n}.$$

Finally, let $M_3 = mM_1M_2$, we have
    \begin{align*}
        \lambda_m(\cR\cap \cC^{\boldsymbol \tau}(N))&\geq \lambda_m(\cR)- M_3\lambda_m(\cR) \sum\limits_{k=n+1}^\infty (N^{-\varepsilon})^{k-n}\\
        &=\left(1-M_3\sum\limits_{j=1}^\infty (N^{-\varepsilon})^{j}\right) \lambda_m(\cR)
    \end{align*}
Since
\[
M_3\sum \limits_{j=1}^\infty (N^{-\varepsilon})^j\to 0
\qquad\text{as}\quad N\to\infty,
\]
we may choose \(N_\kappa\) sufficiently large that
\[
M_3\sum \limits_{j=1}^\infty (N^{-\varepsilon})^j<\kappa
\qquad\text{for all}\,\,N\ge N_\kappa.
\]
It follows that
\[
\lambda_m\bigl(\cR\cap \cC^{\boldsymbol\tau}(N)\bigr)\ge (1-\kappa)\lambda_m(\cR),
\]
as required.
\end{proof}

\vspace{0.5cm}
\section{Step 2: The Set of Leading Rationals}
The set of \textbf{leading rationals} $Q(N,\boldsymbol \tau)$ is defined as follows
	\begin{equation*}
	Q(N,\boldsymbol \tau):=\left\{\frac{\p}{q}\in \Q^m\,:\,\exists L\in \mathcal{L}_n^*\,\,\text{for some}\,\,n,\,\,\text{s.t.}\,\frac{\p}{q}\in L\,\,\text{and}\,\,\begin{cases}\mathrm{i})\,\,N^{n-1}\leq q<N^n\\\mathrm{ii})\,\Delta\left(\frac{\p}{q},\boldsymbol \delta(n)\right)\cap \bigcup\limits_{\cR'\in \mathcal{S}_{n-1}} \cR' \neq\emptyset\end{cases}\right\}
	\end{equation*}
We say that \(\frac{\mathbf p}{q}\) is a \emph{leading rational} of a hyperplane \(L\in \cL_n^*\) if \(L\) is a hyperplane containing \(\frac{\mathbf p}{q}\) and the above conditions are satisfied.

The following observations follow directly from the construction of \(\cC^{\boldsymbol\tau}(N)\).

\begin{enumerate}
\item
Every hyperplane \(L\in \bigcup_{n\in \N}\cL_n^*\) contains at least one leading rational.

\item
Suppose that \(\frac{\mathbf p}{q}\) is a leading rational with $ N^{n-1}\le q< N^n$,
and let \(L\in \cL_n^*\) be a hyperplane containing \(\frac{\mathbf p}{q}\). Let \(\cR\in \V_{(n,\cdots,n)}\) be the rectangle corresponding to \(L\). Recall that the side lengths of \(\cR\) are $t^{-\rho_i}N^{-n\rho_i}$, $i=1,\cdots,m$.
Since \(q<N^n\), we have \(q^{-\rho_i}>N^{-n\rho_i}\), and hence $t^{-\rho_i}N^{-n\rho_i}<\frac12\,q^{-\rho_i}$,
provided \(t>2\). Therefore,
\[
\cR\subseteq
\Delta\!\left(\frac{\mathbf p}{q},\,q^{-\boldsymbol\rho}\right).
\]
Moreover, we have $\delta_i(n) \leq q^{-\rho_i}$, then if $\cR'\in \V_{\boldsymbol{\ell}(n)}$ is such that $\cR'\cap \Delta(L,\boldsymbol \delta(n))\neq\emptyset$, then
	$$ \cR'\subseteq \Delta\left(\frac{\p}{q},2q^{-\boldsymbol \rho}\right).$$
\item Suppose $\frac{\p}{q}\in \Q^m$, $N^{n-1}\leq q < N^n$, is not a leading rational, by Simplex Lemma, $\frac{\p}{q}$ lies on a hyperplane of $\mathcal{L}_n$. But since $\Delta\left(\frac{\p}{q},\boldsymbol \delta(n)\right)\cap \bigcup_{\cR'\in \mathcal{S}_{n-1}} \cR' =\emptyset$, then $\frac{\p}{q}$ must lie in the neighbourhood of a hyperplane in the previous level. That is, there exists $k<n$, $L\in \cL_k^*$ and $\tilde{\cR}\in \V_{\boldsymbol{\ell}(k)}$, such that $\frac{\p}{q}\in \tilde{\cR}$ and $\tilde{\cR} \cap \Delta(L,\boldsymbol \delta(k)) \neq \emptyset$. In particular, any leading rational of $L$ has denominator $v$ such that $v<N^{n-1}\leq q$.
\end{enumerate}

\vspace{0.3cm}
Using the set of leading rationals, we define
\[
\cA_m\bigl(Q(N,\boldsymbol\tau),\Phi_{\boldsymbol\rho}\bigr)
:=
\left\{
\mathbf x\in [0,1]^m :
\left|x_i-\frac{p_i}{q}\right|<\phi_i(q)\ \text{for }i=1,\cdots,m\ \text{for infinitely many } \frac{\mathbf p}{q}\in Q(N,\boldsymbol\tau)
\right\},
\]
where \(\Phi_{\boldsymbol\rho}=(\phi_1,\cdots,\phi_m):\N\to \R_{>0}^m\) is given by
\[
\phi_i(q):=3q^{-\rho_i},\qquad i=1,\cdots,m.
\]
We show that $\cC^{\boldsymbol\tau}(N)\subseteq \cA_m\bigl(Q(N,\boldsymbol\tau),\Phi_{\boldsymbol\rho}\bigr)$.

\begin{lemma}\label{lb:5}
For every $N\in \N$ such that $\cC^{\boldsymbol \tau}(N)$ is defined,
\[
\cC^{\boldsymbol\tau}(N)\subseteq \cA_m\bigl(Q(N,\boldsymbol\tau),\Phi_{\boldsymbol\rho}\bigr).
\]
\end{lemma}

\begin{proof}
Let \(\mathbf x\in \cC^{\boldsymbol\tau}(N)\), and define
\[
P(\mathbf x)
:=
\left\{
\frac{\mathbf p}{q}\in \Q^m\cap [0,1]^m :
\mathbf x\in \prod_{i=1}^m
\left(
\frac{p_i}{q}-q^{-\rho_i},
\frac{p_i}{q}+q^{-\rho_i}
\right)
\right\}.
\]
Since \(\sum_{i=1}^m \rho_i = m+1\) and \(\rho_1\geq \cdots\geq \rho_m\ge 1\), Dirichlet's theorem implies that \(P(\mathbf x)\) is infinite.

If \(P(\mathbf x)\cap Q(N,\boldsymbol\tau)\) is infinite, then \(\mathbf x\in \cA_m(Q(N,\boldsymbol\tau),\Phi_{\boldsymbol\rho})\) immediately, since
\[
\left|x_i-\frac{p_i}{q}\right|<q^{-\rho_i}<3q^{-\rho_i}=\phi_i(q)
\qquad (i=1,\cdots,m).
\]

It therefore remains to consider the case where \(P(\mathbf x)\setminus Q(N,\boldsymbol\tau)\) is infinite. Let $
\frac{\mathbf p}{q}\in P(\mathbf x)\setminus Q(N,\boldsymbol\tau)$ and $n\in \N$
such that $N^{n-1}\leq q<N^n$. Since \(\frac{\mathbf p}{q}\) is not a leading rational, Observation~(3) implies that there exist \(k<n\), a hyperplane \(\widetilde L\in \cL_k^*\), and a rectangle $\widetilde{\cR}\in \V_{\boldsymbol{\ell}(k)}$
such that
\[
\frac{\mathbf p}{q}\in \widetilde{\cR}
\qquad\text{and}\qquad
\widetilde{\cR}\cap \Delta(\widetilde L,\boldsymbol\delta(k))\neq \emptyset.
\]
Thus we may associate \(\frac{\mathbf p}{q}\) to a hyperplane \(\widetilde L\in \bigcup_{j=1}^\infty\cL_j^*\).

Now let \(\frac{\mathbf w}{v}\) be any leading rational of \(\widetilde L\). By Observation~(2), we have
\[
\widetilde{\cR}
\subseteq \Delta\left(\frac{\mathbf{w}}{v},2v^{-\boldsymbol \rho}\right)
\]
Since \(\frac{\mathbf p}{q}\in \widetilde{\cR}\), it follows that
\[
\left|\frac{p_i}{q}-\frac{w_i}{v}\right|<2v^{-\rho_i},
\qquad i=1,\cdots,m.
\]
Therefore, for \(i\in\{1,\cdots,m\}\),
\[
\left|x_i-\frac{w_i}{v}\right|
\le
\left|x_i-\frac{p_i}{q}\right|
+
\left|\frac{p_i}{q}-\frac{w_i}{v}\right|
<
q^{-\rho_i}+2v^{-\rho_i}.
\]
Since \(v< N^{n-1}\le q\), we have \(q^{-\rho_i}\le v^{-\rho_i}\), and hence
\[
\left|x_i-\frac{w_i}{v}\right|<3v^{-\rho_i}=\phi_i(v).
\]
Thus \(\mathbf x\) is \(\Phi_{\boldsymbol\rho}\)-approximable by the leading rational \(\frac{\mathbf w}{v}\).

It remains to show that this happens for infinitely many leading rationals. Suppose, for contradiction, that only finitely many hyperplanes in \(\bigcup_{j=1}^\infty \cL_j^*\) are associated to rationals in \(P(\mathbf x)\setminus Q(N,\boldsymbol\tau)\). Then there exists \(k\in \N\) such that every rational in \(P(\mathbf x)\setminus Q(N,\boldsymbol\tau)\) is associated to a hyperplane in $
\bigcup_{j=1}^k \cL_j^*$.
Hence
\[
P(\mathbf x)\setminus Q(N,\boldsymbol\tau)
\subseteq
\bigcup_{j=1}^k
\bigcup_{L\in \cL_j^*}
\bigcup
\left\{
\cR'\in \V_{\boldsymbol{\ell}(j)} :
\cR'\cap \Delta(L,\boldsymbol\delta(j))\neq \emptyset
\right\}.
\]
For each \(\frac{\mathbf p}{q}\in P(\mathbf x)\), we have
\[
\left|x_i-\frac{p_i}{q}\right|<q^{-\rho_i}\to 0
\qquad\text{as } q\to\infty.
\]
Since \(P(\mathbf x)\setminus Q(N,\boldsymbol\tau)\) is infinite, it follows that \(\mathbf x\) lies in the closure of the set on the right-hand side. On the other hand, by the construction of \(\cC^{\boldsymbol\tau}(N)\),
\[
\mathbf x\notin
\bigcup_{j=1}^k
\bigcup_{L\in \cL_j^*}
\bigcup
\left\{
\cR'\in \V_{\boldsymbol{\ell}(j)} :
\cR'\cap \Delta(L,\boldsymbol\delta(j))\neq \emptyset
\right\}.
\]
Moreover, the latter set is a finite union of closed rectangles, and is therefore closed. Hence it must contain all of its limit points, a contradiction.

We conclude that infinitely many hyperplanes in \(\bigcup_{j=1}^\infty \cL_j^*\) are associated to rationals in \(P(\mathbf x)\setminus Q(N,\boldsymbol\tau)\). Choosing one leading rational from each such hyperplane, we obtain infinitely many leading rationals \(\frac{\mathbf w}{v}\in Q(N,\boldsymbol\tau)\) such that
\[
\left|x_i-\frac{w_i}{v}\right|<3v^{-\rho_i},
\qquad i=1,\cdots,m.
\]
Therefore,
\[
\mathbf x\in \cA_m\bigl(Q(N,\boldsymbol\tau),\Phi_{\boldsymbol\rho}\bigr).
\]
This completes the proof.
\end{proof}

\medskip
Next, we establish another structural property of the set of leading rationals. This is a further technical lemma needed in the proof of the \(T_{G,\cR}\)-Lemma.
\begin{lemma}\label{lb:6}
For any $\frac{\p}{q}\in Q(N,\boldsymbol \tau)$ such that $q\geq N$, there exists $\cR\in \bigcup \limits_{n=1}^\infty \mathcal{S}_{n}$ such that $\cR \subseteq \Delta \left(\frac{\p}{q},q^{-1-\boldsymbol \tau}\right)$ and
\begin{equation}
c_1\lambda_m\left(\Delta\left(\frac{\p}{q},q^{-1-\boldsymbol \tau}\right)\right)\leq \lambda_m(\cR)
\end{equation}
where \(c_1>0\) depends only on the fixed parameters \(N,t,\boldsymbol\tau,\boldsymbol\rho\), and is independent of \(\frac{\mathbf p}{q}\).
\end{lemma}

\begin{proof}
Let
\[
\frac{\mathbf p}{q}\in Q(N,\boldsymbol\tau),
\qquad
N^{n-1}\le q<N^n
\]
for some \(n\geq 2\). By the definition of a leading rational,
\[
\Delta\!\left(\frac{\mathbf p}{q},\boldsymbol\delta(n)\right)
\cap
\bigcup_{\cR\in \cS_{n-1}}\cR
\neq \emptyset.
\]
Choose \(\widetilde{\cR}\in \cS_{n-1}\) such that
\[
\widetilde{\cR}\cap \Delta\!\left(\frac{\mathbf p}{q},\boldsymbol\delta(n)\right)\neq \emptyset.
\]
Since \(\widetilde{\cR}\in \cS_{n-1}\), it belongs to \(\bigcup_{n=1}^\infty \cS_n\), so it suffices to show that \(\widetilde{\cR}\) satisfies the required properties.

The rectangle \(\widetilde{\cR}\) has side lengths $
t^{-\rho_i}N^{-\ell_i(n-1)\rho_i}$, $i=1,\cdots,m$. 
Let \(\mathbf x\in \widetilde{\cR}\cap \Delta\!\left(\frac{\mathbf p}{q},\boldsymbol\delta(n)\right)\). Then for any \(\mathbf z\in \widetilde{\cR}\) and $i\in \{1,\cdots,m\}$,
\[
\left|z_i-\frac{p_i}{q}\right|
\le
|z_i-x_i|+\left|x_i-\frac{p_i}{q}\right|
\le
t^{-\rho_i}N^{-\ell_i(n-1)\rho_i}+\delta_i(n)\]
Recall $\ell_i(n-1) =\left \lceil (n+1)\frac{1+\tau_i}{\rho_i}\right\rceil$, $t>2$ then 
\[
t^{-\rho_i}N^{-\ell_i(n-1)\rho_i}
\le
t^{-\rho_i}N^{-n(1+\tau_i)}<\frac{1}{2}q^{-1-\tau_i}
\]
Similarly,
\[
\delta_i(n)=t^{-\rho_i}N^{-(n+2)(1+\tau_i)}
\le
t^{-\rho_i}N^{-2(1+\tau_i)}N^{-n(1+\tau_i)} < \frac{1}{2} q^{-1-\tau_i}
\]
Thus
\[
\widetilde{\cR}\subseteq \Delta\!\left(\frac{\mathbf p}{q},q^{-1-\boldsymbol\tau}\right).
\]

It remains to estimate the measure of \(\widetilde{\cR}\). We have
\begin{align*}
\lambda_m(\widetilde{\cR})
&=\prod\limits_{i=1}^m t^{-\rho_i} N^{-\ell_i(n-1)\rho_i} = t^{-m-1} \prod\limits_{i=1}^m N^{-\left \lceil (n+1)\frac{1+\tau_i}{\rho_i}\right\rceil \rho_i}\\
&\geq t^{-m-1} \prod\limits_{i=1}^m N^{-(n+1)(1+\tau_i)-\rho_i} = t^{-m-1}N^{-m-1} N^{-2(m+\sum_{i=1}^m \tau_i)}N^{-(n-1)(m+\sum_{i=1}^m \tau_i)}
\end{align*}
On the other hand,
\[
\lambda_m\!\left(\Delta\!\left(\frac{\mathbf p}{q},q^{-1-\boldsymbol\tau}\right)\right)
= \prod\limits_{i=1}^m 2q^{-1-\tau_i} = 2^m q^{-m-\sum_{i=1}^m \tau_i}.
\]
Because \(q\ge N^{n-1}\), we have
\[
q^{-m-\sum_{i=1}^m \tau_i}
\le
N^{-(n-1)(m+\sum_{i=1}^m \tau_i)}.
\]
Therefore,
\[
\lambda_m(\widetilde{\cR})
\ge
2^{-m}t^{-m-1}N^{-m-1} N^{-2(m+\sum_{i=1}^m \tau_i)}
\,
\lambda_m\!\left(\Delta\!\left(\frac{\mathbf p}{q},q^{-1-\boldsymbol\tau}\right)\right).
\]
Thus the conclusion holds with $
c_1:=2^{-m}t^{-m-1}N^{-m-1} N^{-2(m+\sum_{i=1}^m \tau_i)}$, which is independent of \(\frac{\mathbf p}{q}\).
\end{proof}

\vspace{0.5cm}
\section{Step 3: The Set $\D_{\epsilon}(B)$}
In the final step of the article, we construct a Cantor set
\[
\D_\epsilon(B)\subseteq B\cap \cC^{\boldsymbol\tau}(N)\cap \cA_m(\Psi_{\boldsymbol\tau})
\]
together with a mass distribution \(\mu\) supported on \(\D_\epsilon(B)\). We shall verify that \(\mu\) satisfies the hypotheses of the Mass Distribution Principle, thereby completing the proof of Theorem~\ref{tb:2}. The key tool in the construction is the \(T_{G,\cR}\)-lemma, which is inspired by the \(K_{G,B}\)-covering lemma in \cite{BeresnevichVelani2006} and the \(T_{G,I}\)-lemma in \cite{KoivusaloLevesleyWardZhang2024}. Before proving this lemma, we first establish a version of the Vitali covering lemma adapted to our family of rectangles.

\begin{lemma}[Vitali covering lemma for rectangles]\label{lb:7}
Let \(\mathcal V\) be a non-empty collection of rectangles satisfying
\[
\mathcal V\subseteq
\left\{
\Delta\!\left(\frac{\mathbf p}{q},\,3q^{-\boldsymbol\rho}\right)
:
\frac{\mathbf p}{q}\in Q(N,\boldsymbol\tau)
\right\}.
\]
Then there exists a countable subcollection \(\mathcal G\subseteq \mathcal V\) such that:
\begin{enumerate}
\item the rectangles in \(\mathcal G\) are pairwise disjoint;
\item
\[
\bigcup_{\cR\in \mathcal V}\cR
\subseteq
\bigcup_{\cR\in \mathcal G} 5\cR.
\]
\end{enumerate}
\end{lemma}

\begin{proof}
Since \(Q(N,\boldsymbol\tau)\) is countable, the collection \(\mathcal V\) is also countable. Enumerate its elements as
\[
\mathcal V=\{\cR_1,\cR_2,\cR_3,\dots\},
\]
in such a way that if \(\cR_i=\Delta(\frac{\mathbf p}{q},3q^{-\boldsymbol\rho})\) and
\(\cR_j=\Delta(\frac{\mathbf p'}{q'},3{q'}^{-\boldsymbol\rho})\) with \(i<j\), then \(q\le q'\). In other words, the rectangles are ordered from larger to smaller side lengths.

We now construct \(\mathcal G\) greedily. Let \(\cR_{i_1}:=\cR_1\). Having chosen pairwise disjoint rectangles
\[
\cR_{i_1},\dots,\cR_{i_m},
\]
if every rectangle in \(\mathcal V\) intersects one of them, we stop. Otherwise, let \(\cR_{i_{m+1}}\) be the first rectangle in the enumeration that is disjoint from all previously chosen rectangles. In this way we obtain a countable pairwise disjoint subcollection
\[
\mathcal G:=\{\cR_{i_1},\cR_{i_2},\dots\}\subseteq \mathcal V.
\]
Routine calculation may show that $\mathcal{G}$ is a countable subcollection that satsifies the two conditions in the statement of the Lemma. This finishes the proof of the Lemma.
\end{proof}

\medskip
We now state the \(T_{G,\cR}\)-lemma.
\begin{lemma}[\(T_{G,\cR}\)-lemma]\label{lb:8}
Let \(\{R_\ell\}_{\ell\in \N}\) be the collection of rectangles
\begin{equation}\label{eqb:4}
\{R_\ell\}_{\ell\in \N}
=
\left\{
\Delta\!\left(\frac{\mathbf p}{q},\,3q^{-\boldsymbol\rho}\right)
\right\}_{\frac{\mathbf p}{q}\in Q(N,\boldsymbol\tau)},
\end{equation}
arranged so that the denominators of their centres are non-decreasing. When \(N\) is sufficiently large, for every
\[
\cR\in \bigcup_{n=1}^\infty \cS_n
\qquad\text{and}\qquad
G\ge 1,
\]
there exists a finite subcollection
\[
T_{G,\cR}\subseteq \{R_\ell:\ell\ge G\}
\]
such that the rectangles in \(T_{G,\cR}\) are pairwise disjoint, are contained in \(\cR\), and satisfy
\[
\lambda_m\!\left(\bigcup_{R'\in T_{G,\cR}} R'\right)\ge c_2\,\lambda_m(\cR),
\]
where \(c_2>0\) is independent of \(G\) and \(\cR\).
\end{lemma}

\begin{proof}
Observe first that
\[
\limsup_{\ell\to\infty} R_\ell
=
\cA_m\bigl(Q(N,\boldsymbol\tau),\Phi_{\boldsymbol\rho}\bigr).
\]
Let  $\cR\in \cS_n$ for some $n\in \N_{\geq 1}$. 
By Lemma~\ref{lb:4} and Lemma~\ref{lb:5}, if \(N\) is sufficiently large, then
\[
\lambda_m\!\left(\cR\cap \cA_m(Q(N,\boldsymbol\tau),\Phi_{\boldsymbol\rho})\right)
\ge
\lambda_m\!\left(\cR\cap \cC^{\boldsymbol\tau}(N)\right)
\ge
\frac34\,\lambda_m(\cR).
\]
Let \(\cS^*_{n+1}(\cR)\) denote the collection of rectangles in \(\cS_{n+1}\) that are strictly contained in the interior of \(\cR\). For \(N\) sufficiently large, the proportion of level-$(n+1)$ rectangles lost near the boundary of \(\cR\) is arbitrarily small. Together with Lemma \ref{lb:4}, when $N$ is sufficiently large, we may assume that
\[
\lambda_m\!\left(\bigcup_{\cR'\in \cS^*_{n+1}(\cR)}\cR'\right)\ge \frac12\,\lambda_m(\cR).
\]

Fix \(G\ge 1\), and define
\[
\mathcal V
:=
\left\{
R_\ell:\,
\ell\ge G
\ \text{and}\
R_\ell\cap \bigcup_{\cR'\in \cS^*_{n+1}(\cR)}\cR'\neq \emptyset
\right\}.
\]
Since the denominators of the centres of the \(R_\ell\) are non-decreasing, the side lengths of \(R_\ell\) tend to zero as \(\ell\to\infty\). Because every rectangle in \(\cS^*_{n+1}(\cR)\) lies strictly inside \(\cR\), every sufficiently small rectangle meeting \(\bigcup_{\cR'\in \cS^*_{n+1}(\cR)}\cR'\) is actually contained in \(\cR\). By discarding finitely many initial terms if necessary, we may therefore assume that every rectangle in \(\mathcal V\) lies inside \(\cR\).

Applying Lemma~\ref{lb:7} to \(\mathcal V\) we obtain a countable pairwise disjoint subcollection $\mathcal{G}\subseteq \mathcal{V}$, such that
\[
\bigcup_{R\in \mathcal V} R
\subseteq
\bigcup_{R\in \mathcal G} 5R.
\]
Therefore,
\begin{align*}
\lambda_m\!\left(\bigcup_{R\in \mathcal G} 5R\right)
&\ge
\lambda_m\!\left(
\left(\bigcup_{\cR'\in \cS^*_{n+1}(\cR)}\cR'\right)
\cap
\limsup_{\ell\to\infty} R_\ell
\right) \\
&=
\sum_{\cR'\in \cS^*_{n+1}(\cR)}
\lambda_m\!\left(
\cR'\cap \limsup_{\ell\to\infty} R_\ell
\right) \\
&=
\sum_{\cR'\in \cS^*_{n+1}(\cR)}
\lambda_m\!\left(
\cR'\cap \cA_m(Q(N,\boldsymbol\tau),\Phi_{\boldsymbol\rho})
\right) \\
&\ge
\sum_{\cR'\in \cS^*_{n+1}(\cR)}
\frac34\,\lambda_m(\cR')=
\frac34\,
\lambda_m\!\left(\bigcup_{\cR'\in \cS^*_{n+1}(\cR)}\cR'\right)\geq
\frac38\,\lambda_m(\cR).
\end{align*}

Since the rectangles in \(\mathcal G\) are pairwise disjoint, the rectangles \(5R\) with \(R\in\mathcal G\) may overlap, but we still have
\[
\lambda_m\!\left(\bigcup_{R\in\mathcal G}5R\right)
\le
\sum_{R\in\mathcal G}\lambda_m(5R)
=
5^m\sum_{R\in\mathcal G}\lambda_m(R)
=
5^m\,\lambda_m\!\left(\bigcup_{R\in\mathcal G}R\right).
\]
Combining this with the previous estimate gives
\[
\lambda_m\!\left(\bigcup_{R\in\mathcal G}R\right)
\ge
\frac{3}{8\cdot 5^m}\,\lambda_m(\cR).
\]

Finally, since the rectangles in \(\mathcal G\) are pairwise disjoint and \(\lambda_m(\bigcup_{R\in\mathcal G}R)<\infty\), we have
\[
\lambda_m\!\left(\bigcup_{\substack{R_\ell\in \mathcal G\\ \ell\ge j}} R_\ell\right)\to 0
\qquad\text{as }j\to\infty.
\]
Hence there exists \(j_0\ge G\) such that
\[
\lambda_m\!\left(\bigcup_{\substack{R_\ell\in \mathcal G\\ \ell\le j_0}} R_\ell\right)
\ge
\frac12\,
\lambda_m\!\left(\bigcup_{R\in\mathcal G}R\right)
\ge
\frac{3}{16\cdot 5^m}\,\lambda_m(\cR).
\]
We now set
\[
T_{G,\cR}:=\{R_\ell\in \mathcal G:\ell\le j_0\}.
\]
Then \(T_{G,\cR}\) is a finite pairwise disjoint subcollection of \(\{R_\ell:\ell\ge G\}\), every element of \(T_{G,\cR}\) is contained in \(\cR\), and
\[
\lambda_m\!\left(\bigcup_{R'\in T_{G,\cR}}R'\right)\ge \frac{3}{16\cdot 5^m}\,\lambda_m(\cR).
\]
Thus the conclusion holds with $c_2:=3/(16\cdot 5^m)$.
\end{proof}

We now construct the Cantor set \(\D_\epsilon(B)\) and the mass distribution \(\mu\) simultaneously.

\medskip
\noindent\textbf{Construction of \(\D_\epsilon(B)\) and the measure \(\mu\).}
\begin{enumerate}
\item
Fix \(\epsilon>0\) and a ball \(B\subseteq [0,1]^m\). Choose \(N\) sufficiently large so that Lemma~\ref{lb:8} hold, and so that there exists a rectangle \(\cR_0\in \cS_{1}\) with \(\cR_0\subseteq B\). The existence of such an \(N\) is guaranteed by Lemma~\ref{lb:3}. Define
\[
E_0:=\{\cR_0\},
\qquad
\mu(B):=\mu(\cR_0):=1.
\]

\item
Recall the collection \(\{R_\ell\}_{\ell\in \N}\) defined in \eqref{eqb:4}. Let \(G_1\) be such that $d(R_{G_1})\geq N$, and let \(T_{G_1,\cR_0}\) be the collection obtained by applying the \(T_{G,\cR}\)-Lemma to the rectangle \(\cR_0\) with index \(G_1\). By Lemma~\ref{lb:8}, the rectangles in \(T_{G_1,\cR_0}\) are pairwise disjoint and are contained in \(\cR_0\).

\item
For each \(\ell\in \N\), define the \emph{shrink} of
\[
R_\ell=\Delta\left(\frac{\mathbf p}{q},\,3q^{-\boldsymbol\rho}\right)
\]
to be
\[
R_\ell(\boldsymbol\tau):=\Delta\left(\frac{\mathbf p}{q},\,q^{-1- \boldsymbol \tau}\right).
\]
Conversely, given
\[
R(\boldsymbol\tau)=\Delta\left(\frac{\mathbf p}{q},\,q^{-1- \boldsymbol \tau}\right),
\]
we refer to
\[
R=\Delta\left(\frac{\mathbf p}{q},\,3q^{-\boldsymbol\rho}\right)
\]
as its \emph{enlargement}. Define
\[
E_1:=T_{G_1,\cR_0}^{\boldsymbol\tau},
\qquad
T_{G_1,\cR_0}^{\boldsymbol\tau}:=\{R(\boldsymbol\tau):R\in T_{G_1,\cR_0}\}.
\]
Then the rectangles in \(T_{G_1,\cR_0}^{\boldsymbol\tau}\) remain pairwise disjoint and are contained in \(\cR_0\). For each \(R(\boldsymbol\tau)\in E_1\), define
\begin{equation}
\mu\bigl(R(\boldsymbol\tau)\bigr)=\mu(R):=
\frac{\lambda_m(R)}{\sum\limits_{R'\in T_{G_1,\cR_0}}\lambda_m(R')}\,\mu(\cR_0).
\end{equation}

\item
Suppose that \(E_j\) has been constructed for \(j=0,1,\dots,n-1\). We construct the \(n\)-th layer \(E_n\) inductively as follows.
\begin{enumerate}
\item
\textbf{Nesting.}
For each \(\widetilde{R}(\boldsymbol\tau)\in E_{n-1}\), Lemma~\ref{lb:6} provides a rectangle $\widetilde{\cR}\in \bigcup_{j=1}^\infty \cS_j$
such that
\[
\widetilde{\cR}\subseteq \widetilde{R}(\boldsymbol\tau)
\qquad\text{and}\qquad
c_1\,\lambda_m\bigl(\widetilde{R}(\boldsymbol\tau)\bigr)\le \lambda_m(\widetilde{\cR}).
\]
Whenever \(\widetilde{\cR}\) and \(\widetilde{R}(\boldsymbol\tau)\) are related in this way, we write
\[
\widetilde{\cR}\sim \widetilde{R}(\boldsymbol\tau).
\]

\item
Choose \(G_n>G_{n-1}\) such that, whenever \(\ell\ge G_n\), then for every \(\widetilde{R}(\boldsymbol\tau)\in E_{n-1}\),
\begin{equation}\label{eqb:5}
d(\widetilde{R})^{-1-\tau_1} > 3\,d(R_\ell)^{-\rho_m}
\qquad\text{and}\qquad
\frac{\mu\bigl(\widetilde{R}(\boldsymbol\tau)\bigr)}{\lambda_m\bigl(\widetilde{R}(\boldsymbol\tau)\bigr)}
\le d(R_\ell)^\epsilon.
\end{equation}
Such a choice of \(G_n\) is possible because \(E_{n-1}\) is finite and \(d(R_\ell)\to\infty\) as \(\ell\to\infty\). Once \(G_n\) has been chosen, for each \(\widetilde{R}(\boldsymbol\tau)\in E_{n-1}\), apply the \(T_{G,\cR}\)-Lemma to the corresponding rectangle \(\widetilde{\cR}\) with \(\widetilde{\cR}\sim \widetilde{R}(\boldsymbol\tau)\), thereby obtaining a collection of rectangles \(T_{G_n,\widetilde{\cR}}\).

\item
For each such \(\widetilde{\cR}\), define
\[
T_{G_n,\widetilde{\cR}}^{\boldsymbol\tau}
:=
\{R'(\boldsymbol\tau):R'\in T_{G_n,\widetilde{\cR}}\}.
\]
For every \(R(\boldsymbol\tau)\in T_{G_n,\widetilde{\cR}}^{\boldsymbol\tau}\), define
\begin{equation}\label{eqb:6}
\mu\bigl(R(\boldsymbol\tau)\bigr)=\mu(R):=
\frac{\lambda_m(R)}{\sum\limits_{R'\in T_{G_n,\widetilde{\cR}}}\lambda_m(R')}\,
\mu\bigl(\widetilde{R}(\boldsymbol\tau)\bigr).
\end{equation}

\item
We then define the \(n\)-th layer \(E_n\) by
\[
E_n:=
\bigcup_{\substack{\widetilde{\cR}\sim \widetilde{R}(\boldsymbol\tau)\\ \widetilde{R}(\boldsymbol\tau)\in E_{n-1}}}
T_{G_n,\widetilde{\cR}}^{\boldsymbol\tau}.
\]
\end{enumerate}

\item
Finally, define
\begin{equation}
\D_\epsilon(B):=
\bigcap_{n\in \N}\bigcup_{R(\boldsymbol\tau)\in E_n} R(\boldsymbol\tau).
\end{equation}
By \cite[Proposition 1.7]{Falconer1990}, the set function \(\mu\) extends to a Borel measure on \([0,1]^m\) whose support is contained in \(\D_\epsilon(B)\). Moreover, for any set \(F\subseteq [0,1]^m\),
\[
\mu(F)
=
\inf\left\{
\sum_{R(\boldsymbol\tau)\in I}\mu\bigl(R(\boldsymbol\tau)\bigr)
:
F\cap \D_\epsilon(B)\subseteq \bigcup_{R(\boldsymbol\tau)\in I} R(\boldsymbol\tau)
\right\},
\]
where the infimum is taken over all countable subcollections \(I\subseteq \bigcup_{n\in \N}E_n\) that cover \(F\cap \D_\epsilon(B)\). This completes the construction of \(\D_\epsilon(B)\) and the mass distribution \(\mu\).
\end{enumerate}

\vspace{0.3cm}
We make some observation regarding the above construction. These are direct consequences of the $T_{G,\cR}$ Lemma and the construction algorithm.
\begin{enumerate}
\item Suppose $R(\boldsymbol \tau)\in E_n$, $n\geq 1$, then for any $j=1,\cdots, n-1$, there exists a unique rectangle $R'(\boldsymbol\tau)\in E_j$ such that $R(\boldsymbol \tau) \subseteq R'(\boldsymbol \tau)$.  Moreover,  $R(\boldsymbol \tau)\subseteq \cR_0$ for any $R(\boldsymbol \tau)\in E_n$.
\item Suppose $R(\boldsymbol \tau)$ and $R'(\boldsymbol \tau)$ both belong to $E_n$ for some $n\in \N_{\geq 1}$, then either $R(\boldsymbol \tau)=R'(\boldsymbol \tau)$ or $R(\boldsymbol \tau)\cap R'(\boldsymbol \tau) = \emptyset$. The same statement holds if we replace $R(\boldsymbol \tau)$ and $R'(\boldsymbol \tau)$ by $R$ and $R'$.
\item Suppose $R(\boldsymbol\tau)\in E_i$ and $R'(\boldsymbol \tau)\in E_j$ with $i<j$. Then either $R'(\boldsymbol \tau)\subseteq R(\boldsymbol \tau)$ or $R(\boldsymbol \tau)\cap R'(\boldsymbol \tau)= \emptyset$
\end{enumerate}

\vspace{0.3cm}
We now verify that $\mathcal{D}_\epsilon(B)$ is indeed a subset of $B\cap \B_m(\Psi_{\boldsymbol \tau})$. From the construction, it is clear that $\mathcal{D}_\epsilon(B)\subseteq B$, so it remains to show that $\mathcal{D}_{\epsilon}(B)\subseteq \B_m(\Psi_{\boldsymbol \tau})$
\begin{lemma}
$\mathcal{D}_\epsilon(B)\subseteq \cC^{\boldsymbol \tau}(N)\cap \cA_m(\Psi_{\boldsymbol \tau}) \subseteq \B_m(\Psi_{\boldsymbol \tau})$.
\end{lemma}
\begin{proof}
For every $\x\in \mathcal{D}_\epsilon(B)$, one may associate a sequence
	$$R_{i_1}(\boldsymbol \tau)\supseteq \cR_{i_1}\supseteq R_{i_2}(\boldsymbol \tau)\supseteq \cR_{i_2}\supseteq R_{i_3}(\boldsymbol \tau)\supseteq \cR_{i_3}\cdots \ni \x$$
where $R_{i_n}(\boldsymbol \tau) \in E_n$ and $\cR_{i_n}\sim R_{i_n}(\boldsymbol \tau)$. Since $\x\in \bigcap\limits_{n=1}^\infty R_{i_n}(\boldsymbol \tau)$, $\x\in \cA_m(\Psi_{\boldsymbol \tau})$; on the other hand, as $\x\in \bigcap \limits_{i=1}^\infty \cR_{i_n}$, $\x\in \cC^{\boldsymbol \tau}(N)$. This implies $\x\notin \cA_m(\mathbf{c}(N)\Psi_{\boldsymbol \tau})$. Combined, this gives $\x\in \B_m(\Psi_{\boldsymbol \tau})$. 
\end{proof}

We establish one further structural property of the sets \(E_n\). This will be useful later when verifying the hypotheses of the Mass Distribution Principle.

\begin{lemma}\label{lb:9}
Let \(F=B(\mathbf x,r)\subseteq [0,1]^m\) be a ball. Suppose that there exist two distinct rectangles
\[
R(\boldsymbol\tau),\,R'(\boldsymbol\tau)\in E_n
\]
for some \(n\in \N\), both of which intersect \(F\). Then
\[
r\ge d\bigl(R(\boldsymbol\tau)\bigr)^{-\rho_1}.
\]
\end{lemma}

\begin{proof}
Write
\[
R(\boldsymbol\tau)=\Delta\left(\frac{\mathbf p}{q},\,q^{-1-\boldsymbol\tau}\right),
\qquad
R'(\boldsymbol\tau)=\Delta\left(\frac{\mathbf p'}{q'},\,(q')^{-1-\boldsymbol\tau}\right),
\]
and let
\[
R=\Delta\left(\frac{\mathbf p}{q},\,3q^{-\boldsymbol\rho}\right),
\qquad
R'=\Delta\left(\frac{\mathbf p'}{q'},\,3(q')^{-\boldsymbol\rho}\right)
\]
denote their enlargements.

Since \(R(\boldsymbol\tau)\) and \(R'(\boldsymbol\tau)\) belong to \(E_n\), their enlargements \(R\) and \(R'\) are distinct members of the corresponding \(T_{G,\widetilde{\cR}}\)-families used in the construction of \(E_n\). In particular, they are pairwise disjoint. Hence there exists some \(i\in\{1,\cdots,m\}\) such that
\[
\left|\frac{p_i}{q}-\frac{p_i'}{q'}\right|
\ge
3q^{-\rho_i}+3(q')^{-\rho_i}.
\]

Since \(F=B(\mathbf x,r)\) intersects both \(R(\boldsymbol\tau)\) and \(R'(\boldsymbol\tau)\), we have
\[
\left|\frac{p_i}{q}-x_i\right|
\le q^{-1-\tau_i}+r
\qquad\text{and}\qquad
\left|\frac{p_i'}{q'}-x_i\right|
\le (q')^{-1-\tau_i}+r
\]
for \(i=1,\cdots,m\). Therefore, by the triangle inequality, for the index \(i\) above,
\[
2r+q^{-1-\tau_i}+(q')^{-1-\tau_i}
\ge
\left|\frac{p_i}{q}-\frac{p_i'}{q'}\right|
\ge
3q^{-\rho_i}+3(q')^{-\rho_i}.
\]
Rearranging, we obtain
\[
r\ge
\frac12\Bigl(3q^{-\rho_i}+3(q')^{-\rho_i}-q^{-1-\tau_i}-(q')^{-1-\tau_i}\Bigr).
\]
Since \(1+\tau_i-\rho_i>0\), the $q^{-1-\tau_i}<q^{-\rho_i}$ and $(q')^{-1-\tau_i} < (q')^{-\rho_i}$, and likewise for \(q'\). Thus, for \(N\) sufficiently large,
\[
r\ge q^{-\rho_i}.
\]
Since \(\rho_1\geq \cdots\geq \rho_m\), then
\[
r\ge q^{-\rho_1}=d\bigl(R(\boldsymbol\tau)\bigr)^{-\rho_1}.
\]
This completes the proof.
\end{proof}

Finally, we will prove that the mass distribution $\mu$ defined above satisfies the hypothesis of the mass distribution principle with $\alpha = s_{\boldsymbol \rho}-2\epsilon$. More precisely, we establish the following theorem:
\begin{theorem}\label{tb:3}
Let $\mu$ be defined as above. Then there exists a constant $C>0$ and $r_o>0$, such that for every ball $F=\B(\x,r)\subseteq [0,1]^m$, $r<r_o$,
	\begin{equation}\label{eqb:7}
	\mu(F) \leq Cr^{s_{\boldsymbol{\rho}}-2\epsilon},
	\end{equation}
where $s_{\boldsymbol \rho}= \min \{\hat{s}_{\boldsymbol \rho},s\}$ and $\hat{s}_{\boldsymbol \rho}$ is defined in (\ref{eqb:16}).
\end{theorem}

To prove Theorem \ref{tb:3}, we first estimate the $\mu$-measure of rectangles in $E_n$. The result is as follows:
\begin{lemma}\label{lb:10}
Let \(R(\boldsymbol\tau)\in E_n\) for some \(n\geq 1\). Then
\[
\mu\bigl(R(\boldsymbol\tau)\bigr)=\mu(R)\leq C_1\,d(R)^{-m-1+\epsilon},
\]
where the constant \(C_1\) is independent of \(R(\boldsymbol\tau)\).
\end{lemma}
\begin{proof}
Suppose $n\geq 2$, $R(\boldsymbol \tau)\in E_n$, $R(\boldsymbol \tau) = \Delta\left(\frac{\p}{q},q^{-1-\boldsymbol \tau}\right)$. Then $R = \Delta \left(\frac{\p}{q}, 3q^{-\boldsymbol \rho}\right)$, so $\lambda_m(R) = 2^m3^mq^{-m-1}=2^m3^md(R)^{-m-1}$. By Lemma \ref{lb:6} and Lemma \ref{lb:8}, suppose $\tilde{R}(\boldsymbol \tau)$ is the unique rectangle in $E_{n-1}$ such that $R(\boldsymbol \tau)\subseteq \tilde{R}(\boldsymbol \tau)$, then
	$$\sum \limits_{R'\in T_{G_n,\tilde{\cR}}} \lambda_m(R') \geq c_2 \lambda_m (\tilde{\cR}) \geq c_1\cdot c_2 \cdot \lambda_m(\tilde{R}(\boldsymbol \tau)).$$
Directly substituting into the definition, 
    \begin{equation*}
    \begin{split}
	\mu(R(\boldsymbol \tau)) &= \mu(R) = \frac{\lambda_m(R)}{\sum \limits_{R'\in T_{G_n,\tilde{\cR}}} \lambda_m(R')} \times \mu(\tilde{R}(\boldsymbol \tau))\\ &\leq 2^m3^md(R)^{-m-1}\cdot \frac{\mu(\tilde{R}(\boldsymbol \tau))}{c_1\cdot c_2\cdot \lambda_m(\tilde{R}(\boldsymbol \tau))} \overset{(\ref{eqb:5})}\leq C_1d(R)^{-m-1+\epsilon}.
    \end{split}
	\end{equation*}
where $C_1 =2^m3^m/c_1c_2$ is in independent of $R(\boldsymbol \tau)$. Since $E_1$ contains finitely many rectangles, we may further increase $C_1$ such that the statement also holds for $n=1$. This completes the proof of the Lemma.
\end{proof}

\vspace{0.5cm}
Using Lemma \ref{lb:10}, we give a proof of Theorem \ref{tb:3}.
\begin{proof}[Proof of Theorem \ref{tb:3}.]
If $F\cap \mathcal{D}_{\epsilon}(B)=\emptyset$, then $\mu(F)=0$, and inequality (\ref{eqb:7}) holds trivially.  So we may assume that $F\cap \D_\epsilon(B)\neq \emptyset$, which implies $F\cap \bigcap_{n=1}^\infty \bigcup_{R(\boldsymbol \tau)\in E_n} R(\boldsymbol \tau) \neq \emptyset$. In this case, $F$ must have nonempty intersection with some rectangles from each layer (because the layers are nested). Suppose for every $n\geq 1$, $F$ intersects exactly one $R(\boldsymbol \tau)\in E_n$, then $\mu(F)$ is no greater than $\mu$-measure of the unique rectangle it intersects in each layer. By Lemma \ref{lb:10}, if $R(\boldsymbol \tau)\in E_n$, then $\mu(R(\boldsymbol\tau))\rightarrow 0$, as $n\rightarrow \infty$. This again implies $\mu(F)=0$, hence inequality (\ref{eqb:7}) follows. Therefore, it remains to consider the case when $F$ intersects two distinct rectangles $R(\boldsymbol \tau)$ and $R'(\boldsymbol \tau) \in E_n$ for some $n\geq 1$. We fix $n$ to be the smallest such natural number.

Suppose $n=1$, then by Lemma \ref{lb:9}, it is possible to choose $C$ large enough such that the right hand side of (\ref{eqb:7}) is greater than $1$ and the inequality trivially holds. So, we may assume $n\geq 2$. By the minimality of $n$, $F$ intersects a unique rectangle $\hat{R}(\boldsymbol \tau)$ from $E_{n-1}$. Let $Q=d(\hat{R}(\boldsymbol \tau))$, then $\hat{R}(\boldsymbol \tau) = \Delta\left(\frac{\p}{Q},Q^{-1-\boldsymbol \tau}\right)$ for some $\frac{\p}{Q}\in Q(N,\boldsymbol \tau)$. We consider several cases based on the relative size of $r$ and $Q$:
\begin{itemize}
\item \textbf{Case 1:} $r\geq Q^{-1-\tau_m}$. Since the support of $\mu$ lies in $\mathcal{D}_{\epsilon}(B)$, then
$$\mu(F)=\mu(F\cap \mathcal{D}_{\epsilon}(B))\leq \mu(F\cap \hat{R}(\boldsymbol \tau)) \overset{\text{Lemma}\,\ref{lb:10}}\leq C_1d(\hat{R}(\boldsymbol \tau))^{-m-1+\epsilon} = C_1Q^{-m-1+\epsilon}.$$
By assumption $r\geq Q^{-1-\tau_m}$, thus,
	\begin{equation}\label{eqb:8}
	\mu(F)\leq C_1Q^{-m-1+\epsilon} \leq C_1 r^{\frac{m+1-\epsilon}{1+\tau_m}}=C_1r^{\frac{m+1}{1+\tau_m} - \frac{\epsilon}{1+\tau_m}} \leq C_1 r^{\frac{m+1}{1+\tau_m} - \epsilon} \leq C_1r^{s_{\boldsymbol \rho}-2\epsilon}
	\end{equation}
    
\item \textbf{Case 2:} $Q^{-1-\tau_1}\leq r<Q^{-1-\tau_m}$. Let $A=-\frac{\log r}{\log Q}$, then $Q^{-A}=r$ and $A\in (1+\tau_m,1+\tau_1]$. We consider the rectangles in the layer $E_n$ that have nonempty intersection with $F$. By the minimality of $n$, these rectangles must all be contained in  $\hat{R}(\boldsymbol \tau)$. Then
$$\mu(F) \leq \sum\limits_{\substack{R(\boldsymbol \tau)\in E_n\\R(\boldsymbol \tau)\cap F\neq \emptyset}} \mu(R(\boldsymbol \tau)) \overset{(\ref{eqb:6})}= \sum\limits_{\substack{R(\boldsymbol \tau)\in E_n\\R(\boldsymbol \tau)\cap F\neq \emptyset}}\frac{\lambda_m(R)}{\sum\limits_{R'\in T_{G_n},\hat{\mathcal{\cR}}} \lambda_m(R')} \times \mu(\hat{R}(\boldsymbol \tau)).$$
By Lemma \ref{lb:6} and Lemma \ref{lb:8}, $\sum \limits_{R'\in T_{G_n},\hat{\mathcal{R}}} \lambda_m(R') \geq c_1\cdot c_2\cdot \lambda (\hat{R}(\boldsymbol \tau))$. Therefore,
$$\mu(F) \leq \sum\limits_{\substack{R(\boldsymbol \tau)\in E_n\\R(\boldsymbol \tau)\cap F\neq \emptyset}} \lambda_m(R) \times \frac{\mu(\hat{R}(\boldsymbol \tau))}{c_1\cdot c_2 \cdot \lambda_m(\hat{R}(\boldsymbol \tau))} \overset{\text{Lemma}\,\ref{lb:10}}\leq \frac{C_1}{c_1\cdot c_2} \sum\limits_{\substack{R(\boldsymbol \tau)\in E_n\\R(\boldsymbol \tau)\cap F\neq \emptyset}}\lambda_m(R)\cdot \frac{Q^{-m-1+\epsilon}}{Q^{-\sum_{i=1}^m(1+ \tau_i)}}.$$
  We now estimate the summation of the Lebesgue measure. Let $\mathcal{V} = \{R\,:\,R(\boldsymbol \tau)\in E_n,\,R(\boldsymbol \tau)\cap F\neq \emptyset\}$. We observe that the rectangles in $\mathcal{V}$ are disjoint and contained in $\hat{R}(\boldsymbol \tau)$. Let $\nu\in \{1,\cdots,m\}$ be the largest integer such that $1+\tau_{\nu} \geq A$. Then we claim that for any $R\in \V$, 
    $$R\subseteq \prod\limits_{i=1}^{\nu} \left[\frac{p_i}{Q}-Q^{-1-\tau_i}, \frac{p_i}{Q}+Q^{-1-\tau_i}\right] \times \prod\limits_{i=\nu+1}^m [x_i-5r,x_i+5r].$$

Let $R$ be an arbitrary rectangle in $\mathcal{V}$, and $\y=(y_1,\cdots,y_m)\in R$. As $R\subseteq \hat{R}(\boldsymbol \tau)$, then $y_i\in \left[\frac{p_i}{Q}-Q^{-1-\tau_i}, \frac{p_i}{Q}+Q^{-1-\tau_i}\right]$ for any $i\in \{1,\cdots,m\}$. On the other hand, since $R(\boldsymbol \tau)\cap F\neq \emptyset$, $R$ has half the side lengths equal to $d(R)^{-1-\boldsymbol \tau})$, and $R$ has half the side lengths equal to $3d(R)^{-\boldsymbol \rho}$, then an application of triangle inequality gives that for $i\geq \nu+1$,  
$$|y_i-x_i| \leq r+d(R)^{-1-\tau_i}+3d(R)^{-\rho_i}\overset{(\ref{eqb:5})}\leq 5r.$$
This finishes the proof of the claim.

By the disjointness of the rectangles in $\mathcal{V}$, we find
\begin{align*}
    \mu(F) &\leq \frac{C_1}{c_1\cdot c_2} \frac{Q^{-m-1+\epsilon}}{Q^{-\sum_{i=1}^m (1+\tau_i)}} \cdot \lambda_m\left(\prod\limits_{i=1}^{\nu} \left[\frac{p_i}{Q}-Q^{-1-\tau_i}, \frac{p_i}{Q}+Q^{-1-\tau_i}\right] \times \prod\limits_{i=\nu+1}^m [x_i-5r,x_i+5r]\right)\\
    &= C_2\cdot r^{\sum_{i=\nu+1}^m 1}\,\cdot\,Q^{-m-1 + \epsilon + \sum_{i=\nu+1}^m (1+\tau_i)}= C_2\cdot r^{\frac{m+1+\sum_{i=\nu+1}^m (A-1-\tau_i)}{A} - \frac{\epsilon}{A}}\\
    &\leq C_2\cdot r^{\frac{m+1+\sum_{i=\nu+1}^m (A-1-\tau_i)}{A} - \epsilon}
\end{align*}
Here $C_2$ is a constant independent of $F$. We maximize the last term by minimizing the exponent over $A$. Since the exponent is linear in $\frac{1}{A}$ for fixed $\nu$, the minimum of the exponent must take place at the endpoints of the admissible intervals for $A$, i.e., at $1+\tau_i$ for $i\in \{1,\cdots,m\}$. In this case, the exponent is precisely
    $$\frac{m+1+\sum_{k=i}^m(\tau_i-\tau_k)}{1+\tau_i}$$
Therefore
\begin{equation}\label{eqb:10}
	\mu(F) \leq C_2 r^{\min\limits_{1\leq i \leq k} \left\{\frac{m+1+\sum_{k=i}^m (\tau_i-\tau_k)}{1+\tau_i}\right\} - \epsilon}\leq  C_2 r^{s_{\boldsymbol \rho}-2\epsilon}
\end{equation}
	  
\item \textbf{Case 3:} $r<Q^{-1-\tau_1}$.  Let $\V:=\{R\,:\,R(\boldsymbol \tau)\in E_n,\,R(\boldsymbol \tau )\cap F \neq \emptyset\}$ and $Q' := \min \{d(R)\,:\,R \in \V\}$. Let $\tilde{t}\in \Z$ be the smallest positive integer such that $(2^{\tilde{t}}Q')^{-\rho_m} \leq r$; if such positive integer $\tilde{t}$ does not exist, then set $\tilde{t}=-1$. By Lemma \ref{lb:9}, $(Q')^{-\rho_1}\leq r$. Using elementary algebraic manipulation, one can show $\tilde{t}\leq -C_3\log r$ for some constant $C_3$ independent of $F$.  We define a sequence of subcollections of $\V$. We let 
	\begin{align*}
	\V_0&:=\{R\in \V\,:\, d(R)^{-\rho_m}\leq r\}
	\end{align*}
and if $\tilde{t}\geq 1$, we set
	\begin{align*}
	\V_t&:=\{R\in \V\,:\, 2^{t-1}Q' \leq d(R)\leq 2^t Q' \}, \quad \text{for}\,\,t=1,2,\cdots,\tilde{t}-1\\
	\V_{\tilde{t}}&:=\{R\in \V\,:\,2^{\tilde{t}-1} Q' \leq d(R) < r^{-\frac{1}{\rho_m}}\}.
	\end{align*}
It is clear that the union of these $\V_t$'s is $\V$. Since $\mu$ is supported in $\D_\epsilon(B)$, then
	\begin{equation}\label{eqb:9}
	\mu(F) \leq \sum\limits_{R\in \V} \mu(F\cap R) \leq \sum\limits_{R\in \V_0} \mu(R(\boldsymbol \tau)) + \sum\limits_{t=1}^{\tilde{t}} \sum\limits_{R\in \V_t} \mu(F\cap R(\boldsymbol \tau)).
	\end{equation}
We bound the summands on the right of (\ref{eqb:9}).

Firstly, we bound $\sum\limits_{R\in \V_0} \mu(F\cap R)$. By the triangle inequality, it is easy to show that if $R\in \V_0$, then $R\subseteq B(\x,8r)$. Consequently,
\begin{align*}
	 \sum\limits_{R\in \V_0} \mu(R(\boldsymbol \tau))= \sum \limits_{R\in \mathcal{V}_0} \frac{\lambda_m(R)}{\sum\limits_{R'\in T_{G_n,\hat{\cR}}} \lambda_m(R)} \times \mu (\hat{R}(\boldsymbol \tau))&=\frac{\mu(\hat{R}(\boldsymbol \tau))}{\sum\limits_{R'\in T_{G_n,\hat{\cR}}} \lambda_m(R')} \cdot \sum \limits_{R\in \mathcal{V}_0} \lambda_m(R)\\
	 &\leq \frac{C_1}{c_1\cdot c_2} \cdot \frac{Q^{-m-1+\epsilon}}{Q^{-\sum_{i=1}^m (1+\tau_i)}} \cdot \lambda_m(B(\x,8r))\\
    &= C_4 Q^{\sum_{i=1}^m \tau_i-1+\epsilon} r^m
	\end{align*}
Where $C_4$ is a constant independent of $F$. By assumption, $r<Q^{-1-\tau_1}$, so
	\begin{equation}\label{eqb:11}
	\sum\limits_{R\in \V_0} \mu(R(\boldsymbol \tau))\leq C_4r^{\frac{m+1+\sum_{i=1}^m(\tau_1-\tau_i)-\epsilon}{1+\tau_1}} \leq C_4r^{s_{\boldsymbol \rho}-2\epsilon}.
	\end{equation}
 
For $t\in \{1,\cdots,\tilde{t}\}$, define $A_t := -\frac{\log r}{\log (2^tQ')}$, so $(2^tQ')^{-A_t} = r$ and $A_t\in [\rho_m,\rho_1]$ (by Lemma \ref{lb:9}). Let $\nu_1\in \{1,\cdots,m\}$ be the integer such that $A_t\leq \rho_{\nu_1}$ but $A_t > \rho_{\nu_1+1}$, if no such integer exists, set $\nu_1=0$. Let $\nu_2\in \{1,\cdots,m\}$ be the integer such that $A_t\geq 1+\tau_{\nu_2}$ but $A_t< 1+\tau_{\nu_2-1}$, if no such integer exists, set $\nu_2=m+1$. Since $1+\tau_i-\rho_i>0$ for $i=1,\cdots,m$, then $\nu_1 <\nu_2$.

Using triangle inequality, one may show that if $R\in \V_t$, then
    $$R\subseteq \prod\limits_{i=1}^{\nu_1} [x_i-5r,x_i+5r]\times \prod\limits_{i=\nu_1+1}^m[x_i-5(2^{t-1}Q')^{-\rho_i}, x_i+5(2^{t-1}Q')^{-\rho_i}]$$
As the rectangles in $\V_t$ are pair wise disjoint and have comparable volume $\asymp$ $(2tQ')^{-\sum_{i=1}^m \rho_i}$, a standard volume argument gives
    \begin{equation}\label{eqb:12}
        \#\V_t \leq C_5 \prod\limits_{i=1}^{\nu_1} r(2^{t-1}Q')^{\rho_i}= C_5 r^{\sum_{i=1}^{\nu_1} 1} (2^{t-1}Q')^{\sum_{i=1}^{\nu_1} \rho_i}
    \end{equation}

For each $R\in \V_t$, we bound the measure $\mu(F\cap R(\boldsymbol \tau))$ using rectangles in level $E_{n+1}$. Let $\V_{R}:=\{\tilde{R}\in E_{n+1}\,:\,\tilde{R}\subseteq R(\boldsymbol \tau),\,\tilde{R}\cap F\neq \emptyset\}$. Then
	\begin{equation}\label{eqb:13}
	\begin{split}
	\mu(F\cap R(\boldsymbol \tau))\leq \sum\limits_{\tilde{R}\in \V_R} \mu(\tilde{R})&\leq \frac{C_1}{c_1c_2} d(R)^{-m-1+\sum_{i=1}^m (1+\tau_i)+\epsilon} \sum\limits_{\tilde{R}\in \V_R} \lambda_m(\tilde{R}).\\
	&\leq \frac{C_1}{c_1c_2} (2^tQ')^{-m-1+\sum_{i=1}^m (1+\tau_i)+\epsilon}\sum \limits_{\tilde{R}\in \V_{R}}\lambda_m(\tilde{R})
	\end{split}
	\end{equation}
By condition  (\ref{eqb:5}) and Lemma \ref{lb:9}, $d(R')^{-\rho_m} \leq d(R)^{-1-\tau_1}<d(R)^{-\rho_1}\leq r$, which by triangle inequality, shows that $\bigcup \V_R\subseteq 5F\cap R(\boldsymbol \tau)$. Therefore, by the disjointness of the rectangles in level $E_{n+1}$, there exists a constant $C_6$ independent of $F$ such that
    \begin{align} 
    \sum \limits_{\tilde{R}\in \V_{R}}\lambda_m(\tilde{R}) \leq \prod\limits_{i=1}^m \min\{r, (2^{t-1}Q')^{-1-\tau_i}\}\leq C_6 \prod\limits_{i=1}^{\nu_2-1} (2^tQ')^{-1-\tau_i} \cdot \prod\limits_{i=\nu_2}^m r
    \end{align}
 Together with (\ref{eqb:12}) and (\ref{eqb:13}), we conclude that there exists a constant $C_7$ independent of $F$, such that 
	\begin{align*}
	\sum \limits_{R\in \V_t} \mu(F\cap R(\boldsymbol \tau)) &\leq C_7 (2^tQ')^{-m-1+\sum_{i=1}^m (1+\tau_i)+\epsilon} \prod\limits_{i=1}^{\nu_1} r(2^{t}Q')^{\rho_i}\cdot \prod\limits_{i=1}^{\nu_2-1} (2^tQ')^{-1-\tau_i} \cdot \prod\limits_{i=\nu_2}^m r\\
    &=C_7 r^{\sum_{i=1}^{\nu_1} 1 + \sum_{i=\nu_2}^m 1}\cdot (2^tQ')^{-\sum_{i=\nu_1+1}^m \rho_i + \sum_{i=\nu_2}^m (1+\tau_i)} \cdot (2^tQ')^{\epsilon} 
	\end{align*}
Since $(2^tQ')^{-A_t}=r$, $A_t\geq 1$ then
    \begin{align*}
    \sum \limits_{R\in \V_t} \mu(F\cap R(\boldsymbol \tau)) &\leq C_7 r^{\sum_{i=1}^{\nu_1} 1 + \sum_{i=\nu_2}^m 1} \cdot (r)^{\sum_{i=\nu_1+1}^m \frac{\rho_i}{A_t} - \sum_{i=\nu_2}^m \frac{1+\tau_i}{A_t}}\cdot r^{-\epsilon}\\
    &\leq C_7 r^{\sum_{i=1}^{\nu_1} 1 + \sum_{i=\nu_2}^m \frac{A_t-(1+\tau_i)+\rho_i}{A_t} + \sum_{i=\nu_1+1}^{\nu_2-1} \frac{\rho_i}{A_t}}\\
    &= C_7 r^{\sum_{i=1}^{\nu_1} 1 + \sum_{i=\nu_2}^m 1 - \sum_{i=\nu_2}^m \frac{1+\tau_i-\rho_i}{A_t} + \sum_{i=\nu_1+1}^{\nu_2-1} \frac{\rho_i}{A_t}}
    \end{align*}
We maximize last expression by minimizing the exponent of $r$. Since the exponent is linear in $A_t$ once $\nu_1$ and $\nu_2$ is fixed. Then the minimum should take place at the endpoints of the admissible interval for $1/A_t$, which is precisely at $1+\tau_i$ or $\rho_i$ for some $i\in \{1,\cdots,m\}$. In this case, by comparing to the expression (\ref{eqb:16}), we see that 
    $$r^{\sum_{i=1}^{\nu_1} 1 + \sum_{i=\nu_2}^m 1 - \sum_{i=\nu_2}^m \frac{1+\tau_i-\rho_i}{A_t} + \sum_{i=\nu_1+1}^{\nu_2-1} \frac{\rho_i}{A_t}}\leq r^{s_{\boldsymbol \rho}}.$$
Returning to (\ref{eqb:9}), we find
	\begin{align*}
	\mu(F) &\leq \sum\limits_{R\in \V_0} \mu(R(\boldsymbol \tau)) + \sum\limits_{t=1}^{\tilde{t}} \sum\limits_{R\in \V_t} \mu(F\cap R(\boldsymbol \tau))\\
	&\leq C_4r^{s_{\boldsymbol \rho}-2\epsilon} + \sum\limits_{t=1}^{\tilde{t}} C_7r^{s_{\boldsymbol \rho}-\epsilon}\\
	&\leq C_4r^{s_{\boldsymbol \rho}-2\epsilon}+ (-C_3\log r) \cdot C_7r^{s_{\boldsymbol \rho}-\epsilon}.
	\end{align*}
There exists $r_o$ such that whenever $r<r_o$, $-\log r \leq r^{-\epsilon}$. Therefore, there exists a constant $C_8$ independent of $F$, such that 
	\begin{equation}\label{eqb:14}
	\mu(F) \leq C_8r^{s_{\boldsymbol \rho}-2\epsilon}
	\end{equation}
\end{itemize}

\vspace{0.3cm}
Finally, from (\ref{eqb:8}), (\ref{eqb:10}) and (\ref{eqb:14}), we may conclude that there exists a constant $C$ and $r_o$, such that whenever, $F=B(\x,r)\subseteq [0,1]^m$, $r<r_o$, then
	$$\mu(F) \leq Cr^{s_{\boldsymbol \rho}-2\epsilon}.$$
This finishes the proof of Theorem \ref{tb:3} and by the mass distribution principle \ref{pb:1}, Theorem \ref{tb:2} follows accordingly.
\end{proof}

\vspace{1cm}
\bibliographystyle{amsplain}
\bibliography{references} 

\end{document}